\documentclass [12pt,reqno]{amsart}
\usepackage{times}
\usepackage{mathtools}
\usepackage[toc,page]{appendix} 
\usepackage{bm}
\usepackage{pifont}
\allowdisplaybreaks
\numberwithin{equation}{subsection}

\usepackage{amssymb,latexsym,amsmath,amsfonts, eucal}
\usepackage[usenames,dvipsnames,svgnames,table]{xcolor}
\DeclareMathOperator{\rng}{rng}

\newtheorem{thm}{Theorem}[section]
      \newtheorem{lemma}[thm]{Lemma}
      
      \newtheorem{cor}[thm]{Corollary}
      
      \newtheorem{example}[thm]{Example}
      
      \newtheorem{rmk}[thm]{Remark}
      
      \numberwithin{equation}{section}
\title [ vector valued de Branges spaces]{ Vector valued de Branges spaces, cnu contractions and functional models}

\author[Garg]{Bharti Garg}

\address{
	Department of Mathematics\\
	Indian Institute of Technology Ropar\\
	140001\\
	India}

\email{ {bharti.20maz0012@iitrpr.ac.in},{bhartigargfdk@gmail.com}}

\author[Sarkar]{Santanu Sarkar}

\address{
	Department of Mathematics\\
	Indian Institute of Technology Ropar\\
	140001\\
	India}

\email{ {santanu@iitrpr.ac.in},{ santanu87@gmail.com}}

\begin{document}


\begin{abstract}
In this paper, we study vector valued de Branges spaces associated with a de Branges operator, defined as a pair of Fredholm operator valued analytic functions on a domain symmetric with respect to the unit circle. Using a suitable direct sum decomposition of a Hilbert space, we construct a class of vector valued reproducing kernel Hilbert spaces and show that under some assumptions these are vector valued de Branges spaces. We further demonstrate that these spaces provide  functional models for certain class of completely non-unitary contraction operators. We also give a Fredholm-type criterion for verifying the hypotheses of the main construction and apply it to several concrete classes of completely non-unitary contractions. Next, we establish connections between the Sz.-Nagy-Foias characteristic function of the contraction operator, the projection operator valued function arising from the Hilbert space decomposition, and the reproducing kernel of the de Branges space. In particular, we show that the characteristic function coincides with the projection operator valued function on the unit disc. Enroute, we also obtain a complete unitary invariance of a certain class of cnu contractions in terms of de Branges quotient operator valued functions. Finally, we discuss certain aspects of the canonical contraction in de Branges model and its $L^2$ realization. These results provide a new perspective on the role of vector valued de Branges spaces in operator model theory.
\end{abstract}
\maketitle

\tableofcontents
\section{Introduction}
\label{Section 1}
L. de Branges introduced the Hilbert spaces of entire functions through three fundamental axioms, which are now known as de Branges spaces. Motivated by the classical scalar valued Paley-Wiener spaces, de Branges studied spaces of scalar valued entire functions, beginning with his seminal work \cite{brange} and applied it to problems in inverse spectral theory for canonical differential systems. A comprehensive treatment can be found in his monograph \cite{Brange}. These spaces were later generalized to the vector valued set up of entire and meromorphic functions in his subsequent works (see \cite{Branges 1}, \cite{Branges 2}, \cite{Branges 3}, \cite{Branges 4}). 
These spaces of entire functions have close connections with the M. G. Krein's theory of entire operators with finite and equal, as well as infinite, deficiency indices (see \cite{kreinlect}).

The de Branges spaces of $\mathbb{C}^n$ valued entire functions corresponding to matrix valued reproducing kernels have been studied extensively by Arov and Dym in their work. These spaces are found to have applications to various analytical problems, including direct and inverse problems for canonical differential systems and Dirac-Krein systems, prediction theory for stationary Gaussian processes and inverse spectral problems for Feller-Krein string equation (see \cite{DM76}). A comprehensive treatment can be found in their excellent monographs (see \cite{ArD08}, \cite{multi}). The connection between these de Branges spaces and M. G. Krein's class of entire operators with finite and equal deficiency indices was established in \cite{sarkar}. Recently, Derkach and Dym studied the theory of entire symmetric and isometric operators in rigged de Branges-Pontryagin spaces of $\mathbb{C}^n$ valued entire functions (see \cite{dd2}, \cite{dd}, \cite{dd3}, \cite{dd4}). Dym in \cite{DymJFA} investigated the two classes of vector valued de Branges spaces, namely $\mathcal{H}(U)$ and $\mathcal{B}(\mathfrak{E})$ spaces in the framework of matrix valued reproducing kernels over both the upper half-plane and the unit disc. More recently, Dhara and Dym in \cite{DD} established the connections between de Branges spaces of $\mathbb{C}^n$ valued entire functions in the setting of the open upper half plane and the open unit disc, with the solutions of truncated matrix Hamburger moment problem and the truncated matrix trigonometric moment problem, respectively. 

In the present paper, we study the vector valued de Branges spaces $\mathcal{B}(\mathfrak{E})$ based on a de Branges operator, which is a pair of Fredholm operator valued analytic functions defined on a domain which is symmetric with respect to unit circle and contains the open unit disc. 
Our main objective is to investigate how $\mathcal{B}(\mathfrak{E})$ spaces serve as functional models for certain completely non-unitary (cnu) contraction operators. The corresponding spaces with respect to the upper half plane were investigated in \cite{mahapatra}, where it was observed that these spaces serve as functional models for a M. G. Krein's class of entire operators with infinite deficiency indices. We refer to \cite{mahapatra2}, where these spaces were shown to generalize vector valued Paley-Wiener spaces and where further aspects of their structure were also discussed. Further connections with the quasi-Lagrange type interpolation were recently obtained in \cite{mahapatra3}. We also refer to \cite{GS}, where the relationship between another class of vector valued de Branges spaces $\mathcal{H}(U)$ corresponding to operator valued reproducing kernel and a class of simple, closed, densely defined, symmetric operators with infinite deficiency indices is explored in the setting where domain is symmetric with respect to the real line. The primary motivation of our objectives comes from M. G. Krein's theory of entire operators, and the extension theory of Straus. Also, we want to see how these fit with the Sz.-Nagy-Foais theory for certain cnu contractions. See for instance, \cite{fundamentalkrein}, \cite{Straus123}, \cite{Straus}, \cite{Nagy-Foais}. 

\subsection{Organization of the paper}
The paper is organized as follows. Section \ref{Section 1} contains the introduction, including a brief historical background, the organization of the paper, and the notations used throughout. The preliminary material required for the paper, including the basic definitions of vector valued RKHS, cnu contraction operators $T$ acting on a complex separable infinite dimensional Hilbert space $H$, and the vector valued de Branges spaces $\mathcal{B}(\mathfrak{E})$ corresponding to operator valued reproducing kernel, is presented in Section \ref{Section 2}. In this section, we also discuss a Hardy space characterization of $\mathcal B(\mathfrak{E})$ spaces in the disc setting. 
In Section \ref{Section 3}, we establish a direct sum decomposition of the Hilbert space $\mathcal{H}= H \oplus H$ into the components $(V_0-zI)(\ker V)^{\perp}$ and the infinite dimensional closed subspace $((V_0-aI)(\ker V)^{\perp})^{\perp}$, for all $z$ in the domain $\Omega_{a}$ defined in \eqref{Omega}. Here, the operators $V_0$ and $V$ denote the isometry and partial isometry, respectively, constructed from $T$ as given in \eqref{V0} and \eqref{V}, and $a$ is a resolvent point of $T$ on the unit circle. The idea and motivation of this decomposition comes from the \textit{Straus extension} of symmetric operators. For reference, see \cite{Straus}.
In Section \ref{Section 4}, we construct an abstract vector valued RKHS $\mathbf{H}$ using this direct sum decomposition of $\mathcal{H}$ and show that this space is a vector valued de Branges space associated with a de Branges operator $(E_-(z), E_+(z))$. We further prove that the de Branges space $\mathbf{H}$ constructed in this way serves as functional model for the isometric operator $V_0$, which leads to a functional model for the cnu contraction operator $T$. We also give a Fredholm-type criterion for verifying the hypotheses of the main construction and apply it to several concrete classes of cnu contractions, including diagonal essentially unitary contractions, non-normal block weighted-shift-type operators, and compressed shifts on scalar and vector valued model spaces.
In Section \ref{Section 5}, we establish the connections between the Sz.-Nagy-Foias characteristic function of $T$, the $((V_0-aI)(\ker V)^{\perp})^{\perp}$-valued projection operator $P_Y(z)$ used to construct the de Branges space $\mathbf{H}$ and the corresponding de Branges operator valued reproducing kernel. In particular, we show that the operator valued functions $\Theta_{T}(z)$ and $P_Y(z)$ coincide on $\mathbb{D}$. We also prove that, for the class of cnu contractions considered in this paper, unitary equivalence is completely characterized by the coincidence of the corresponding operator valued functions \(E_+^{-1}E_-\) arising from the associated de Branges operators.
In the final section, we study a canonical contraction associated with the de Branges operator obtained from the preceding construction. We first prove that the quotient $\Theta_\mathfrak E=E_+^{-1}E_-$
is purely contractive. Under the additional assumption that the de Branges operator is entire, we show that the Sz.-Nagy--Foias characteristic function of this canonical contraction coincides with \(\Theta_\mathfrak E\). Finally, we describe the larger \(L^2\)-ambient space, which gives a unitary dilation picture for this canonical contraction and connects the construction with the Arov--Dym framework.

\subsection{Motivation of the paper}

The motivation for the present work comes from the interaction between two
important model-theoretic frameworks in operator theory. On one hand, de Branges
spaces and their vector valued analogues provide reproducing kernel Hilbert space
models for certain classes of symmetric, isometric and entire operators. In this
direction, the works of M. G. Krein, de Branges, Arov, Dym and others show that the
structure of such operators can be encoded by analytic operator valued functions
and by the corresponding reproducing kernels. On the other hand, the classical
Sz.-Nagy--Foias theory provides a functional model for cnu
contraction operators in terms of their characteristic functions. Although both
theories describe operators through analytic data, the precise connection between
these two points of view is not immediate, especially in the setting of vector
valued de Branges spaces over domains symmetric with respect to the unit circle.

The purpose of this paper is to make this connection explicit for a natural class
of cnu contractions. Starting with a cnu contraction \(T\) having
at least one resolvent point on the unit circle, we associate with \(T\) an isometry
\(V_{0}\) and a partial isometry \(V\) acting on \( H\oplus  H\).
The existence of a point \(a\in \rho(T)\cap \mathbb T\) allows us to construct a
Straus-type direct sum decomposition involving the spaces
\[
        (V_{0}-zI)(\ker V)^{\perp}
        \quad \text{and} \quad
        \big((V_{0}-aI)(\ker V)^{\perp}\big)^{\perp}.
\]
This decomposition is the main geometric ingredient of the paper. It gives rise to
a projection operator valued analytic function \(P_{Y}(z)\), and hence to a vector
valued RKHS of analytic functions.

Our aim is to realize a class of such cnu contractions within the framework of vector valued de
Branges spaces. We prove that under some assumptions, the RKHS obtained from the above decomposition is a vector valued de Branges space \(\mathcal{B}(\mathfrak{E})\) associated
with a suitable de Branges operator
\[
       \mathfrak{E}=(E_{-},E_{+}).
\]
Thus the contraction \(T\) is realized as a compressed multiplication operator on a
de Branges space. In this way, the construction gives a de Branges space
realization of certain cnu contractions and provides a concrete bridge between the
Straus extension method and the Sz.-Nagy--Foias model theory.

A further motivation is to understand how the analytic data obtained from the
de Branges construction is related to the classical characteristic function of \(T\).
We show that the projection operator valued function \(P_{Y}(z)\), which arises
naturally from the direct sum decomposition, coincides on the unit disc with the
Sz.-Nagy--Foias characteristic function \(\Theta_{T}(z)\), after the natural
identifications of the corresponding spaces. Hence the characteristic function is
recovered from the de Branges-space construction. This shows that the de Branges
operator constructed in the paper contains the same model-theoretic information as
the characteristic function.

Consequently, for the class of cnu contractions considered here, unitary
equivalence can be characterized in terms of the corresponding de Branges
operators through the operator valued function \(E_{+}^{-1}E_{-}\). This places
vector valued de Branges spaces within the functional model theory of cnu
contractions and also provides a systematic source of examples.

\subsection{Notations}
The following notations will be used throughout the paper: 
\begin{itemize} 
\item $\mathbb{C},$ $\mathbb{T}$, and $\mathbb{D}$ denote the complex plane, the unit circle, and the unit disc, respectively. 
\item  $I$ denotes the identity operator on some Hilbert space.
\item $H$ denotes a complex separable infinite dimensional Hilbert space.
\item $B(H)$ denotes the space of all bounded linear operators on Hilbert space $H$.
\item $K(H)$ denotes the space of all compact linear operators on Hilbert space $H$.
\item $\rho_w(z)= 1-z\bar{w}$.
\item $B_{\varepsilon}(a)$ denotes the open ball of radius $\varepsilon$ centered at $a$ in $\mathbb{C}$, defined by
\[
B_{\varepsilon}(a)=\{z\in\mathbb{C}:|z-a|<\varepsilon\}.
\]
\item $C_{\varepsilon}(a)$ denotes the arc on the unit circle $\mathbb{T}$ obtained by intersecting the open ball $B_{\varepsilon}(a)$ with $\mathbb{T}$, i.e.,
\[
C_{\varepsilon}(a)
=
B_{\varepsilon}(a)\cap \mathbb{T}
=
\{z\in \mathbb{T}: |z-a|<\varepsilon\}.
\]
\item For an operator $A$;
\begin{itemize}
\item [(i)]$A^*$ denotes the adjoint operator.
\item [(ii)]$A\succeq 0$ denotes that $A$ is positive semi-definite.
\item [(iii)] $\mathrm{rng} (A)$, $\ker (A)$, and $\mathcal{D}(A)$  denote the range, kernel, and domain of $A$, respectively.
\item [(iv)] $\rho(A)$ denotes the resolvent set of $A$.
\item[(v)] A point $\alpha$ is said to be a point of regular type for $A$ if there exists a positive constant $c_\alpha$ such that $$\Vert (A-\alpha I )g\Vert \geq c_\alpha \Vert g \Vert \hspace{0.3cm} \text{for all} \hspace{0.2cm} g \in \mathcal{D}(A).$$
\end{itemize}
\item $R_z$ denotes the generalized backward shift operator of $H$-valued functions and is defined by \begin{equation}
    (R_z g)(\xi) := \left\{
   \begin{array}{ll}
         \frac{g(\xi)-g(z)}{\xi-z}  & \mbox{if }~ \xi \neq z \vspace{0.1cm} \\ 
        g'(z) & \mbox{if }~ \xi = z
   \end{array} \right.
\end{equation}
for every $z,~ \xi \in \mathbb{C}$.
\end{itemize}
\section{Preliminaries}
\label{Section 2}
In this section, we recall all the basic definitions and preliminaries required for the paper.
\\
\textbf{I.~Reproducing kernel Hilbert spaces and positive kernels:}

First, we recall the definition of vector valued RKHS and of positive kernel functions.

A Hilbert space $\mathcal{H}$ of $H$-valued functions defined on a nonempty set $\Omega \subseteq \mathbb{C}$ is called a reproducing kernel Hilbert space (RKHS) if there exists a $B(H)$-valued function $K_{\omega}(\lambda)$ defined on $\Omega \times \Omega$ satisfying the following properties:

\begin{itemize}
\item For every $u \in H$ and $\omega \in \Omega$, the function $K_{\omega}u$ belongs to $\mathcal{H}$.
\item For every $f \in \mathcal{H}$, $u \in H$, and $\omega \in \Omega$, the reproducing property holds:
\[
\langle f, K_{\omega}u \rangle_\mathcal{H} 
= 
\langle f(\omega), u \rangle_H.
\]
\end{itemize}
The function $K_{\omega}(\lambda)$ is uniquely determined and is called the reproducing kernel (RK) of the Hilbert space $\mathcal{H}$. Let $\delta_{\omega}$ denote the evaluation operator at $\omega$, defined by $\delta_{\omega}(f)=f(\omega)$. Then the reproducing kernel admits the representation
\[
K_{\omega}(\lambda)= \delta_{\lambda}\delta_{\omega}^*.
\]
A function $K:\Omega \times \Omega \to B(H)$ is called a positive kernel if, for any $n \in \mathbb{N}$, any choice of points $\omega_1,\ldots,\omega_n \in \Omega$, and vectors $u_1,\ldots,u_n \in H$, the inequality
\[
\sum_{i,j=1}^{n} 
\langle K_{\omega_j}(\omega_i)u_j , u_i \rangle 
\ge 0
\]
holds.
It follows directly from the definition that the reproducing kernel of an RKHS is always a positive kernel. Moreover, the vector valued version of Moore's theorem (see \cite[Theorem 6.12]{Paulsen}) states that every positive kernel determines a unique RKHS of vector valued functions for which it acts as the reproducing kernel.

For a detailed discussion on RKHS, see \cite{Aronszajn 1950} and \cite{Paulsen}.
\\
\textbf{II.~Contraction operators, defect spaces and characteristic functions:}

A bounded linear operator $T$ on $H$ is called a contraction operator if $\Vert T \Vert \leq 1$. A contraction $T$ on $H$ is called completely non-unitary (cnu) if there does not exist a non-zero closed subspace reducing $T$ such that restriction of $T$ to that subspace is unitary.  The defect operators of $T$ are defined by
\[
D_T = (I-T^*T)^{1/2}
\quad \mbox{and} \quad
D_{T^*} = (I-TT^*)^{1/2},
\]
and the defect spaces associated with $T$ are defined by
\[
\mathcal{D}_T = \overline{\rng(D_T)}
\quad \mbox{and} \quad
\mathcal{D}_{T^*} = \overline{\rng(D_{T^*})}.
\]
The characteristic operator valued function of $T$ is defined for $z \in \mathbb{D}$ by
\begin{equation}\label{char}\Theta_T(z):= -T+z D_{T^*}(I-zT^*)^{-1}D_T \big|_{\mathcal{D}_T}.\end{equation}
The function $\Theta_T(z)$ is an analytic operator valued function from $\mathcal{D}_T$ into $\mathcal{D}_{T^*}$.
\\
\textbf{III.~Fredholm operators and analytic Fredholm theorem:}

Next, we recall the notion of Fredholm operators that will be used in the description of vector valued de Branges space \(\mathcal{B}(\mathfrak{E})\). A bounded linear operator \( A \in  B(H) \) is called Fredholm if  \( \dim(\ker(A)) < \infty \), \( \dim(\ker(A^*)) < \infty \) and $\mathrm{rng}  (A) \) is closed in \( H \). 
The index associated with a Fredholm operator is defined by  
\begin{equation}
\text{index}(A) = \dim(\ker(A)) - \dim(\ker(A^*)).
\end{equation}

\begin{thm} \label{Theorem 2.5}
If $A \in  B(H)$ is a Fredholm operator. Then the following assertions are true.
\begin{itemize}
\item[(1)] \( A^* \) is a Fredholm operator and $ \text{index}(A) = - \text{index}(A^*). $
\item[(2)] \( A \) is invertible iff $$ \text{index}(A) = 0 \quad \text{and} \quad \ker(A) \ (\text{or} \ \ker(A^*)) = 0. $$
\end{itemize}
\end{thm}

Let us also recall the Fredholm analytic theorem that will be used below.
\begin{thm}(see \cite[Theorem 3.3]{fritz2}, \cite[XI, Corollary 8.4]{GGK}) Let $T : \Omega \rightarrow B(\mathfrak{X})$ be a Fredholm operator valued analytic function, where $\Omega$ is an open connected subset of $\mathbb{C}$. Then, exactly one of the following two conditions holds true: \begin{itemize}\item $T(\lambda)$ is not boundedly invertible for any $\lambda \in \Omega$. \item $T(\lambda)^{-1}$ is an analytic Fredholm operator valued function for all $\lambda \in \Omega \setminus D$, where $D$ in $\Omega$ is a discrete set (a countable set with no accumulation point in $\Omega$) and $T(\lambda)^{-1}$ is meromorphic on $\Omega$.\end{itemize}\end{thm}

\textbf{IV.~Vector valued de Branges spaces associated with de Branges operators:}

We now define the class of vector valued de Branges spaces, denoted by $\mathcal{B}(\mathfrak{E})$. Here we consider a domain $\Omega$ which is symmetric with respect to the unit circle and contains the open unit disc. A de Branges operator with respect to the unit disc is defined as a pair of operator valued analytic functions $\mathfrak{E}(\lambda)=(E_-(\lambda), E_+(\lambda))$ whose components satisfy the following conditions:
\begin{itemize}
\item[(1)] $E_+$, $E_- : \Omega \rightarrow B(H)$ both are Fredholm operators for all $\lambda \in \Omega$.
\item[(2)] $E_+$ and $E_- $ both are invertible at least at one point of $\Omega$.
\item[(3)] $E_+^{-1}E_-$ satisfy the following conditions:
\begin{eqnarray*}
(E_+^{-1}E_-)^*(\lambda)(E_+^{-1}E_-)(\lambda) &\preceq& I 
\quad \text{for all } \lambda \in \mathbb{D}, \\
(E_+^{-1}E_-)^*(\lambda)(E_+^{-1}E_-)(\lambda) &=& I 
\quad \text{for all } \lambda \in \Omega\cap \mathbb{T},\\
(E_+^{-1}E_-)(\lambda)(E_+^{-1}E_-)^*(\lambda) &=& I 
\quad \text{for all } \lambda \in \Omega\cap \mathbb{T}.
\end{eqnarray*}
\end{itemize}
We recall that the definition and construction of de Branges operator with respect to the upper half plane $\mathbb{C}_+$, where the domain $\Omega =\mathbb{C}$ were studied in \cite{mahapatra, mahapatra2}.
Corresponding to a de Branges operator with respect to the unit disc, we define the kernel function by \begin{equation} \label{Equation 2.4}
K_\xi^{\mathfrak{E}}(z):= \left\{
    \begin{array}{ll}
         \frac{E_+(z)E_+(\xi)^*-E_-(z)E_-(\xi)^*}{\rho_\xi(z)}  & \mbox{if } z,~\xi \in \Omega~\mbox{and}~z\overline{\xi} \neq 1 \vspace{0.2cm}\\
         \frac{E_+^{'} (\frac{1}{\overline{\xi}})E_+(\xi)^*- E_-^{'}(\frac{1}{\overline{\xi}})E_-(\xi)^*}{- \overline{\xi}} & \mbox{if }\xi \in \Omega~\mbox{and}~ z \overline{\xi}=1.
    \end{array} \right.
\end{equation} 
The kernel is positive on  $\Omega \times \Omega$, and $\mathcal{B}(\mathfrak{E})$ denotes the corresponding unique RKHS of $H$-valued analytic functions on $\Omega$. The vector valued de Branges spaces $\mathcal{B}(\mathfrak{E})$ were studied in \cite{mahapatra, mahapatra2} as Hilbert spaces of entire functions.

For clarity, we want to remark the following regarding the above definition of a de Branges operator. At first glance, condition~(3) is unproblematic at those points where $E_+$ is invertible, since the expression $E_+^{-1}E_-$ is then defined in the usual sense. The only possible ambiguity arises at points where $E_+$ fails to be invertible. In what follows, we show that $E_+^{-1}E_-$ nevertheless admits a well-defined interpretation at such points in $\mathbb{D}\cup(\Omega\cap\mathbb{T})$, so that condition~(3) is meaningful on the entire set under consideration. In particular, the stated inequalities and equalities in the respective domains continue to hold there as well.

By conditions~(1) and~(2), the analytic operator valued function $E_+$ is Fredholm for every $\lambda \in\Omega$ and invertible at least at one point of $\Omega$. Therefore, by the Fredholm analytic theorem, $E_+$ is invertible at every point of $\Omega$ except possibly on a discrete subset. Moreover, at each point of this discrete set where $E_+$ is not invertible, the inverse $E_+^{-1}$ has a pole. Since $E_-$ is analytic on $\Omega$, it follows that the operator valued function $E_+^{-1}E_-$ is analytic on $\Omega$ except possibly at a discrete set, where it may have either poles or removable singularities.

We now show that, at those points of the discrete set in $\mathbb{D} \cup (\Omega \cap \mathbb{T})$, the function $E_+^{-1}E_-$ has only removable singularities. Let $\lambda_0\in\mathbb{D}$ be an isolated singularity of $E_+^{-1}E_-$. Then $E_+^{-1}E_-$ is analytic in some deleted neighbourhood of $\lambda_0$. On this deleted neighbourhood, $E_+$ is invertible, and hence the contractivity condition~(3) is valid on this deleted neighbourhood, i.e.,
\[
\bigl\|(E_+^{-1}E_-)(\lambda)\bigr\|\leq 1
\qquad
\text{for all } \lambda \text{ in a deleted neighbourhood of } \lambda_0.
\]
Therefore, by Riemann's theorem on removable singularities for operator valued functions \cite[Theorem~1.10.3]{OTAA 192}, $\lambda_0$ is a removable singularity of $E_+^{-1}E_-$. Redefining the function at $\lambda_0$ by continuity, we obtain
\[
\bigl\|(E_+^{-1}E_-)(\lambda_0)\bigr\|
=
\lim_{\lambda\to\lambda_0}\bigl\|(E_+^{-1}E_-)(\lambda)\bigr\|
\leq 1.
\]
Hence, $E_+^{-1}E_-$ extends analytically in a neighbourhood of $\lambda_0$ and remains contractive there. Since $\lambda_0$ was arbitrary, it follows that $E_+^{-1}E_-$ extends analytically in a neighbourhood of every such singularity point in $\mathbb{D}$, and the contractive inequality continues to hold throughout the unit disc.

It remains to consider the singularity points of $E_+^{-1}E_-$ lying on $\Omega\cap\mathbb{T}$. Let $\lambda_0\in \Omega\cap\mathbb{T}$ be such an isolated singularity. Since $E_+^{-1}E_-$ is meromorphic on $\Omega$, the point $\lambda_0$ is either a pole or a removable singularity.

We claim that $\lambda_0$ cannot be a pole. Indeed, let $\lambda\to \lambda_0$ radially with $\lambda\in\mathbb{D}$. Then, by the contractivity condition in ~(3),
\[
\bigl\|(E_+^{-1}E_-)(\lambda)\bigr\|\leq 1.
\]
Thus, $E_+^{-1}E_-$ remains bounded as $\lambda\to\lambda_0$ radially through points of the unit disc. Therefore, by \cite[Theorem~1.10.4]{OTAA 192}, the point $\lambda_0$ cannot be a pole. Consequently, $\lambda_0$ is a removable singularity of $E_+^{-1}E_-$.

Hence, $E_+^{-1}E_-$ extends analytically across a neighbourhood of $\lambda_0$. We continue to denote this extension by $E_+^{-1}E_-$. We now show that the equalities in condition~(3) remain valid at $\lambda_0$. Since the set of singularity points of $E_+^{-1}E_-$ is discrete in $\Omega$, there exists a sequence $\{\lambda_n\}\subset (\Omega\cap\mathbb{T})\setminus \{\text{singularity points of }E_+^{-1}E_-\}$ such that $\lambda_n\to\lambda_0$. For each $n$, condition~(3) gives
\[
(E_+^{-1}E_-)^*(\lambda_n)(E_+^{-1}E_-)(\lambda_n)=I
\]
and
\[
(E_+^{-1}E_-)(\lambda_n)(E_+^{-1}E_-)^*(\lambda_n)=I.
\]
Since $E_+^{-1}E_-$ is continuous at $\lambda_0$, passing to the limit as $n\to\infty$ yields
\[
(E_+^{-1}E_-)^*(\lambda_0)(E_+^{-1}E_-)(\lambda_0)=I
\]
and
\[
(E_+^{-1}E_-)(\lambda_0)(E_+^{-1}E_-)^*(\lambda_0)=I.
\]
Thus, the extended value $(E_+^{-1}E_-)(\lambda_0)$ is unitary. In particular,
\[
\bigl\|(E_+^{-1}E_-)(\lambda_0)\bigr\|=1.
\]

Since $\lambda_0$ was arbitrary, it follows that every isolated singularity of $E_+^{-1}E_-$ on $\Omega\cap\mathbb{T}$ is removable, and after analytic extension at such points, the equalities in condition~(3) continue to hold on $\Omega\cap\mathbb{T}$. Therefore, condition~(3) is well defined on the whole of $\mathbb{D}\cup(\Omega\cap\mathbb{T})$.

For the study of complex function theory on operator valued analytic functions, see \cite{OTAA 192, hille}.

We also record the following observation concerning the positivity of the kernel function \eqref{Equation 2.4}. Since $E_-$ is Fredholm for every $\lambda \in\Omega$ and invertible at at least one point of $\Omega$, the Fredholm analytic theorem implies that $E_-$ is invertible at every point of $\Omega$ except possibly on a discrete subset. Moreover, by condition~(3), the operator valued function $E_+^{-1}E_-$ is unitary on $\Omega\cap\mathbb{T}$ at all points where it is defined.

Now consider the operator valued function
\[
\lambda \longmapsto \left\{\bigl(E_+^{-1}E_-\bigr)\!\left(\frac{1}{\overline{\lambda}}\right)^*\right\}^{-1}.
\]
This function is well defined and analytic on $\Omega$ except possibly on a discrete subset. Let $D$ denote the union of the corresponding discrete exceptional sets. Then both the operator valued functions
\[
(E_+^{-1}E_-)(\lambda)
\qquad\text{and}\qquad
 \left\{\bigl(E_+^{-1}E_-\bigr)\!\left(\frac{1}{\overline{\lambda}}\right)^*\right\}^{-1}
\]
are analytic on the domain $\Omega\setminus D$. Further, for every $\lambda\in (\Omega\setminus D) \cap\mathbb{T}$, condition~(3) implies that
\[
(E_+^{-1}E_-)(\lambda)
=
\left\{\bigl(E_+^{-1}E_-\bigr)\!\left(\frac{1}{\overline{\lambda}}\right)^*\right\}^{-1}.
\]
Since $(\Omega\setminus D) \cap\mathbb{T}$ has an accumulation point, the identity theorem for operator valued functions \cite[Theorem 3.11.5]{hille} yields
\[
(E_+^{-1}E_-)(\lambda)
=
\left\{\bigl(E_+^{-1}E_-\bigr)\!\left(\frac{1}{\overline{\lambda}}\right)^*\right\}^{-1}
\]
on the domain $\Omega\setminus D$.

Using this identity, one obtains the positivity of the kernel function \eqref{Equation 2.4} on the dense subset
\[
(\Omega\setminus D)\times (\Omega\setminus D)\subset \Omega\times\Omega.
\]
Since the kernel function is continuous on $\Omega\times\Omega$, it follows that the kernel is positive on the whole of $\Omega\times\Omega$. For a detailed proof of the positivity of the corresponding kernel in the upper half plane setup, we refer the reader to \cite[Sections~4--5]{mahapatra}.
\\
\textbf{V. Hardy space characterization of $\mathcal B(\mathfrak E)$ in the disc setting:}

Let \(\mathfrak{X}\) be a Hilbert space. We denote by
\(H^2_{\mathfrak{X}}(\mathbb D)\) the \(\mathfrak{X}\)-valued Hardy space on
the unit disc, that is,
\[
H^2_{\mathfrak{X}}(\mathbb D)
=
\left\{
f(z)=\sum_{n=0}^{\infty}y_nz^n:
y_n\in\mathfrak{X},\ 
\sum_{n=0}^{\infty}\|y_n\|_{\mathfrak{X}}^2<\infty
\right\}.
\]
The norm is given by
\[
\|f\|^2_{H^2_{\mathfrak{X}}(\mathbb D)}
=
\sum_{n=0}^{\infty}\|y_n\|_{\mathfrak{X}}^2.
\]
Equivalently,
\[
\|f\|^2_{H^2_{\mathfrak{X}}(\mathbb D)}
=
\sup_{0<r<1}
\int_{\mathbb T}\|f(r\zeta)\|_{\mathfrak{X}}^2\,dm(\zeta),
\]
where \(m\) denotes the normalized Lebesgue measure on \(\mathbb T\).

We also define the exterior Hardy space \(H^2_{\mathfrak{X}}(\mathbb D_e)\) by
\[
H^2_{\mathfrak{X}}(\mathbb D_e)
=
\left\{
g(z)=\sum_{n=1}^{\infty}y_{-n}z^{-n}:
y_{-n}\in\mathfrak{X},\ 
\sum_{n=1}^{\infty}\|y_{-n}\|_{\mathfrak{X}}^2<\infty
\right\},
\]
where
\[
\mathbb D_e=\{z\in\mathbb C: |z|>1\}\cup\{\infty\}.
\]
Thus the functions in \(H^2_{\mathfrak{X}}(\mathbb D_e)\) are analytic on
\(\{z:|z|>1\}\), vanish at infinity, and have square-summable negative
Fourier coefficients. The norm is
\[
\|g\|^2_{H^2_{\mathfrak{X}}(\mathbb D_e)}
=
\sum_{n=1}^{\infty}\|y_{-n}\|_{\mathfrak{X}}^2.
\]
On boundary values, \(H^2_{\mathfrak{X}}(\mathbb D_e)\) is identified with 
\[ H^2_{\mathfrak X,-} := L^2_{\mathfrak X}(\mathbb T)\ominus H^2_{\mathfrak X}(\mathbb D), \] where \[ H^2_{\mathfrak X,-} = \left\{ \sum_{n<0}y_n\zeta^n: \sum_{n<0}\|y_n\|_{\mathfrak X}^2<\infty \right\}. \]
Let \(\Theta:\mathbb D\to B(\mathfrak{X})\) be a Schur class operator valued
function, that is,
\[
\Theta(z)^*\Theta(z)\leq I,
\qquad z\in\mathbb D.
\]
The de Branges--Rovnyak space associated with \(\Theta\), denoted by
\(\mathcal H(\Theta)\), is the reproducing kernel Hilbert space with kernel
\[
K^\Theta_\xi(z)
=
\frac{I-\Theta(z)\Theta(\xi)^*}{1-z\overline{\xi}},
\qquad z,\xi\in\mathbb D.
\]

The following result gives the disk analogue of the classical Hardy space description of de Branges spaces. In the scalar upper half-plane setting, this description goes back to de Branges; see \cite{Brange}. Also, see \cite{DymAdv}. For de Branges spaces of $\mathbb{C}^n$ valued entire functions corresponding to matrix valued reproducing kernels in the upper half-plane, see Arov and Dym \cite[Theorem~3.10]{multi}. In the present paper, the functions are defined on a domain symmetric with respect to the unit circle, and therefore the natural formulation is in terms of the unit disk and the exterior disk. We include the proof for the sake of keeping the paper self-contained. The proof follows the same kernel-transport idea as in the above references.

\begin{thm}\label{Theorem HSC}
Let \(\mathfrak{E}=(E_-,E_+)\) be a de Branges operator on a domain \(\Omega\) symmetric
with respect to the unit circle and containing \(\mathbb D\). Put
\[
\Theta(z)=E_+(z)^{-1}E_-(z),
\qquad z\in\mathbb D,
\]
where the removable singularities of \(E_+^{-1}E_-\) are understood in the
sense explained above. Then the following assertions hold.

\begin{enumerate}
\item 
\[
f\in \mathcal B(\mathfrak{E})
\quad\Longleftrightarrow\quad
E_+^{-1}f\in\mathcal H(\Theta)
\]
and
\[
\|f\|_{\mathcal B(\mathfrak{E})}
=
\|E_+^{-1}f\|_{\mathcal H(\Theta)}.
\]
Moreover, for \(f,h\in \mathcal B(\mathfrak{E})\),
\[
\langle f,h\rangle_{\mathcal B(\mathfrak{E})}
=
\left\langle
E_+^{-1}f,
E_+^{-1}h
\right\rangle_{\mathcal H(\Theta)}.
\]

\item If, in addition, $\Omega = \mathbb{C}$ and \(\Theta\) is such that
\[
\Theta(\zeta)^*\Theta(\zeta)
=
\Theta(\zeta)\Theta(\zeta)^*
=
I
\quad \text{for all } \zeta\in\mathbb T,
\]
then
\[
        f\in \mathcal B(\mathfrak E)
        \quad\Longleftrightarrow\quad
        E_+^{-1}f\big|_{\mathbb D}\in H_\mathfrak X^2(\mathbb D)
        \quad\text{and}\quad
        E_-^{-1}f\big|_{\mathbb D_e}\in H_\mathfrak X^2(\mathbb D_e).
\]
In this case,
\[
\|f\|_{\mathcal B(\mathfrak{E})}
=
\|E_+^{-1}f\|_{H^2_{\mathfrak{X}}(\mathbb D)},
\]
and
\[
\langle f,h\rangle_{\mathcal B(\mathfrak{E})}
=
\int_{\mathbb T}
\left\langle
E_+(\zeta)^{-1}f(\zeta),
E_+(\zeta)^{-1}h(\zeta)
\right\rangle_{\mathfrak{X}}
\,dm(\zeta).
\]
\end{enumerate}
\end{thm}

\begin{proof}
Here,
\[
\Theta(z)=E_+(z)^{-1}E_-(z),
\qquad z\in\mathbb D.
\]
By the definition of a de Branges operator, after removing the possible removable
singularities, \(\Theta\) is a Schur class operator valued function on \(\mathbb D\).
Hence
\[
K^\Theta_\xi(z)
=
\frac{I-\Theta(z)\Theta(\xi)^*}{1-z\overline{\xi}}
\]
is a positive kernel on \(\mathbb D\), and it defines the de Branges--Rovnyak space
\(\mathcal H(\Theta)\).

Since
\[
E_-(z)=E_+(z)\Theta(z),
\]
we obtain, for \(z,\xi\in\mathbb D\),
\[
\begin{aligned}
K_\xi^\mathfrak{E}(z)
&=
\frac{
E_+(z)E_+(\xi)^*
-
E_-(z)E_-(\xi)^*
}
{1-z\overline{\xi}}                                      \\[4pt]
&=
\frac{
E_+(z)E_+(\xi)^*
-
E_+(z)\Theta(z)\Theta(\xi)^*E_+(\xi)^*
}
{1-z\overline{\xi}}                                      \\[4pt]
&=
E_+(z)
\left(
\frac{I-\Theta(z)\Theta(\xi)^*}{1-z\overline{\xi}}
\right)
E_+(\xi)^*                                               \\[4pt]
&=
E_+(z)K^\Theta_\xi(z)E_+(\xi)^*.
\end{aligned}
\]
Thus the reproducing kernel of \(\mathcal B(\mathfrak E)\), is obtained
from the de Branges--Rovnyak kernel by multiplication by \(E_+(z)\) on the left and
by \(E_+(\xi)^*\) on the right. This gives a unitary
identification
\[
M_{E_+}:\mathcal H(\Theta)\longrightarrow \mathcal B(\mathfrak E),
\qquad
M_{E_+}g=E_+g.
\]
Therefore,
\[
f\in \mathcal B(\mathfrak E)
\quad\Longleftrightarrow\quad
E_+^{-1}f\in\mathcal H(\Theta),
\]
where \(E_+^{-1}f\) is understood on the set where \(E_+\) is invertible and then
by removable analytic continuation. The norm and inner product are
therefore
\[
\|f\|_{\mathcal B(\mathfrak E)}
=
\|E_+^{-1}f\|_{\mathcal H(\Theta)}
\]
and
\[
\langle f,h\rangle_{\mathcal B(\mathfrak E)}
=
\left\langle
E_+^{-1}f,
E_+^{-1}h
\right\rangle_{\mathcal H(\Theta)}.
\]
This proves the first assertion. Now, to prove the second assertion, let
\[
g=E_+^{-1}f.
\]
Then
\[
g\in H^2_{\mathfrak{X}}(\mathbb D)
\ominus
\Theta H^2_{\mathfrak{X}}(\mathbb D)
\]
if and only if
\[
g\perp \Theta q
\quad \text{for every } q\in H^2_{\mathfrak{X}}(\mathbb D).
\]
In terms of boundary values, this means
\[
\int_{\mathbb T}
\left\langle
g(\zeta),\Theta(\zeta)q(\zeta)
\right\rangle_{\mathfrak{X}}
\,dm(\zeta)
=
0
\]
for every \(q\in H^2_{\mathfrak{X}}(\mathbb D)\). Equivalently,
\[
\int_{\mathbb T}
\left\langle
\Theta(\zeta)^*g(\zeta),q(\zeta)
\right\rangle_{\mathfrak{X}}
\,dm(\zeta)
=
0
\]
for every \(q\in H^2_{\mathfrak{X}}(\mathbb D)\). Hence
\[
\Theta^*g\in H^2_{\mathfrak{X},-}.
\]
Using the boundary identification of $
H^2_{\mathfrak{X},-}$ with $ H^2_{\mathfrak{X}}(\mathbb D_e)$,
we get
\[
\Theta^*g\in H^2_{\mathfrak{X}}(\mathbb D_e).
\]
Since,
\[
\Theta(\zeta)^{-1}=\Theta(\zeta)^*
\quad \text{for all } \zeta\in\mathbb T.
\]
Also,
\[
E_-=E_+\Theta.
\]
Therefore, on boundary values,
\[
E_-^{-1}f
=
\Theta^{-1}E_+^{-1}f
=
\Theta^*g.
\]
Thus
\[
E_-^{-1}f\in H^2_{\mathfrak{X}}(\mathbb D_e).
\]
We have shown that, in this case, $f\in \mathcal B(\mathfrak E)$
if and only if
\[
        E_+^{-1}f\big|_{\mathbb D}\in H_\mathfrak X^2(\mathbb D) \quad \text{and} \quad E_-^{-1}f\big|_{\mathbb D_e}\in H_\mathfrak X^2(\mathbb D_e).
\]
This proves the Hardy-space characterization.

Finally, the norm is the usual Hardy norm of \(E_+^{-1}f\):
\[
\|f\|_{\mathcal B(\mathfrak E)}
=
\|E_+^{-1}f\|_{H^2_{\mathfrak{X}}(\mathbb D)}.
\]
Therefore, for \(f,h\in \mathcal B(\mathfrak E)\),
\[
\langle f,h\rangle_{\mathcal B(\mathfrak E)}
=
\int_{\mathbb T}
\left\langle
E_+(\zeta)^{-1}f(\zeta),
E_+(\zeta)^{-1}h(\zeta)
\right\rangle_{\mathfrak{X}}
\,dm(\zeta).
\]
This completes the proof.
\end{proof}

\section{Direct sum decomposition of $H \oplus H$}
\label{Section 3}
Let $T \in B(H)$ be a cnu contraction. Throughout the paper, we also take the additional hypothesis that a point $a \in \mathbb{T}$ belongs to the resolvent set $\rho(T)$ of $T$. 
Consider the operator $V_0 : H \oplus \{0\} \to H \oplus H$ by
\begin{equation}\label{V0}
V_0=
\begin{bmatrix}
T & 0 \\
D_T & 0
\end{bmatrix}.
\end{equation}

Then $V_0$ is an isometry. Indeed, for $h \in H$,
\[
V_0(h,0) = (Th, D_T h),
\]
and therefore
\[
\|V_0(h,0)\|^2
= \|Th\|^2 + \|D_T h\|^2
= \|Th\|^2 + \langle (I - T^{*}T)h, h\rangle
= \langle h,h\rangle
= \|h\|^2.
\]
Thus $V_0^{*}V_0 = I$ on $H \oplus \{0\}$, and hence $V_0$ is an isometry. 
Moreover, since $T$ is cnu, it follows that $V_0$ is also cnu. We also consider the operator $V : H \oplus H \to H \oplus H$ by
\begin{equation}\label{V}
V
=
\begin{bmatrix}
T & 0 \\
D_T & 0
\end{bmatrix}.
\end{equation}

Then $V$ is a contraction and, moreover, a partial isometry and a dilation of $T$. Indeed,
\[
V(h_1,h_2) = (Th_1, D_T h_1),
\]
and therefore the initial space of $V$ is
\[
(\ker V)^{\perp} = H \oplus \{0\}.
\]
Thus $V$ acts as an isometry on $(\ker V)^{\perp}$ and vanishes on $\{0\} \oplus H$. 
Since the restriction of $V$ to its initial space coincides with the cnu isometry $V_0$, it follows that $V$ is also cnu.

Moreover,
\[
\ker V^*
=
\left\{
\begin{bmatrix}
D_{T^*} g\\
- T^* g
\end{bmatrix}
: g \in H
\right\}.
\] 

\begin{rmk} The operator $V$ is not invertible and therefore $0\in\sigma(V)$. 
We determine for which nonzero $\lambda$ the operator $V-\lambda I$ fails to be invertible. By \cite[Problem 71]{Halmos},
\[
V-\lambda I=
\begin{bmatrix}
T-\lambda I & 0\\
D_T & -\lambda I
\end{bmatrix}
\]
is invertible if and only if
\[
\begin{bmatrix}
T-\lambda I & 0\\
0 & -\lambda I
\end{bmatrix}
\]
is invertible. The latter holds if and only if $T-\lambda I$ is invertible (since $\lambda\neq0$). Consequently,
\[
\lambda\notin\sigma(V)\quad\Longleftrightarrow\quad \lambda\notin\sigma(T),\ \lambda\neq0.
\]
Hence
\[
\sigma(V)=\sigma(T)\cup\{0\}.
\]
\end{rmk}

\begin{rmk}\label{Remark 3.2}
We have assumed that $T$ is a cnu contraction with a point $a \in \rho(T)$. By the above remark, it follows that $a \in \rho({V})$. Hence there exists a constant $c>0$ such that $$\Vert(V-aI)(f+g) \Vert \geq c \Vert(f+g) \Vert$$ for all $f+g \in \mathcal{H}=H \oplus H$. In particular, for every $f \in H \oplus \{0 \}$, $$c \Vert f\Vert \leq \Vert(V-aI)f \Vert=\Vert(V_0-aI)f \Vert. $$ Thus $a$ is a point of regular type of $V_0$. Note that this implies that                                                                                                                                                                                                                                                                                                                                                                                                                               $(V_0-aI)(\ker V)^{\perp}$ is closed. In fact, the following result holds: if $z \in \mathbb{C}$ is not an eigenvalue of $T$, then the closedness of $(V_0-aI)(\ker V)^{\perp}$ is both necessary and sufficient for $z$ to be a point of regular type of $T$. Moreover, the set of points of regular type is always open. For the proof , see \cite[Chapter 1, Section 3]{kreinlect}
\end{rmk}

\begin{rmk}\label{Remark 3.3}
Since $T$ is a cnu contraction, $T^*$ is also a cnu contraction. Hence, any point $z \in \mathbb{T}$ is not an eigenvalue of $T^*$ as well. Thus, \begin{eqnarray*}
& & \ker (T^*-zI)= \{ 0 \} \\
&\Rightarrow & \overline{\rng (T-zI)} = H.
\end{eqnarray*}
Now, if the point $z$ is also a point of regular type of $T$, then $(T-zI)$ is injective and has closed range. This implies that $z \in \rho(T)$.
\end{rmk}

In the forthcoming theorem, we establish a direct sum decomposition of the Hilbert space $\mathcal{H}= H \oplus H$, where the second component is a fixed infinite dimensional closed subspace of $\mathcal{H}$.
\begin{thm}\label{Theorem D}
Let the operators $T$, $V$, $V_0$ be as defined above, and assume $a \in \mathbb{T}$ belongs to the resolvent set $\rho(T)$. Then for every $z \in \mathbb{C}\setminus \mathbb{T}$, the following direct sum decomposition holds:
\begin{eqnarray}
\mathcal{H}&=& (V_0-zI)(\ker V)^{\perp} \dotplus ((V_0-aI)(\ker V)^{\perp})^{\perp} \label{D}\\ \nonumber
&=& (V-zI)(\ker V)^{\perp} \dotplus ((V-aI)(\ker V)^{\perp})^{\perp}
\end{eqnarray}
\end{thm}

\begin{proof}
Consider an extension $U_a$ of $V_0$ defined by
\[
\mathcal{D}(U_a)
= (\ker V)^{\perp} 
  \dotplus 
  \big((V_0-aI)(\ker V)^{\perp}\big)^{\perp},
\]
and
\[
U_a(f+\phi)=V_0 f + a\phi,
\]
where $f \in (\ker V)^{\perp}$ and 
$\phi \in \big((V_0-aI)(\ker V)^{\perp}\big)^{\perp}$.
The proof is carried out in several steps. In Step~1, we show that
\[
(\ker V)^{\perp} 
\cap 
\big((V_0-aI)(\ker V)^{\perp}\big)^{\perp}
= \{0\}.
\]
In Step~2, we prove that $U_a$ is an isometry on its domain. 
Steps~3 and~4 establish that $\mathcal{D}(U_a)=\mathcal{H}$ and 
$\rng(U_a)=\mathcal{H}$, respectively. 
Finally, we prove that the direct sum decomposition holds.\\
\textbf{Step 1.}
Let $h \in (\ker V)^{\perp} \cap 
\big((V_0-aI)(\ker V)^{\perp}\big)^{\perp}$ 
such that $h \neq 0$. Then
\[
0 = U_a(0) = U_a(h-h) = V_0 h - a h = (V_0-aI)h.
\]
This is a contradiction, since $a$ is a point of regular type for $V_0$.\\
\textbf{Step 2.}
Let $f,g \in (\ker V)^{\perp}$ and 
$\phi,\psi \in \big((V_0-aI)(\ker V)^{\perp}\big)^{\perp}$. 
Then
\begin{eqnarray*}
\langle U_a(f+\phi),U_a(g+\psi) \rangle
&=& \langle V_0 f + a\phi,V_0 g + a\psi\rangle \\
&=& \langle V_0 f,V_0 g\rangle 
   + |a|^2 \langle \phi,\psi\rangle
   + \bar a \langle V_0 f,\psi\rangle
   + a \langle \phi,V_0 g\rangle .
\end{eqnarray*}
Since $V_0$ is isometric on $(\ker V)^{\perp}$, we have
\[
\langle V_0 f,V_0 g\rangle = \langle f,g\rangle .
\]
Moreover, since $\psi \in \big((V_0-aI)(\ker V)^{\perp}\big)^{\perp}$, we have
\[
0=\langle (V_0-aI)f,\psi\rangle,
\]
which implies
\[
\langle V_0 f,\psi\rangle = a\langle f,\psi\rangle .
\]
Similarly,
\[
\langle \phi,V_0 g\rangle = \bar a \langle \phi,g\rangle .
\]
Therefore,
\begin{eqnarray*}
\langle U_a(f+\phi),U_a(g+\psi) \rangle
&=& \langle f,g\rangle + \langle \phi,\psi\rangle 
   + |a|^2 \langle f,\psi\rangle 
   + |a|^2 \langle \phi,g\rangle \\
&=& \langle f+\phi,\, g+\psi\rangle .
\end{eqnarray*}
\textbf{Step 3.}
Observe that
\begin{eqnarray*}
(\mathcal{D}(U_a))^{\perp}
&=& \big((\ker V)^{\perp} + ((V_0-aI)(\ker V)^{\perp})^{\perp} \big)^{\perp} \\
&\overset{\mbox{\ding{172}}}{=}& (\ker V)^{\perp\perp} 
   \cap \big((V_0-aI)(\ker V)^{\perp}\big)^{\perp\perp} \\
&\overset{\mbox{\ding{173}}}{=}& \ker V 
   \cap (V_0-aI)(\ker V)^{\perp} \\
&=& (\{0\} \oplus H) 
   \cap \left\{
   \begin{bmatrix}
   (T-a)h \\
   D_T h
   \end{bmatrix}
   : h \in H
   \right\} \\
&\overset{\mbox{\ding{174}}}{=}& \{0\}.
\end{eqnarray*}
The equality $\mbox{\ding{172}}$ follows from the fact that if $M$ and $N$ are closed subspaces of a Banach space, then
\[
(M+N)^{\perp}=M^{\perp}\cap N^{\perp}.
\]
The equality $\mbox{\ding{173}}$ holds because $\ker(V)$ is closed and 
$(V_0-aI)(\ker V)^{\perp}$ is also closed, since $a$ is a point of regular type of $V_0$.

For $\mbox{\ding{174}}$, suppose there exists a nonzero $g \in H$ such that
\[
\begin{bmatrix}
0 \\
g
\end{bmatrix}
\in (\{0\}\oplus H) 
\cap 
\left\{
\begin{bmatrix}
(T-a)h \\
D_T h
\end{bmatrix}
: h \in H
\right\}.
\]
Then $(T-aI)h=0$, which implies $h=0$ since $a \in \rho(T)$. Consequently, $g=D_T h=0$, a contradiction.
 
It remains to justify that $\mathcal D(U_a)$ is closed. Put $M_a=(V_0-aI)(\ker V)^\perp $.
As computed later in equation~\eqref{eqn 5.6}, we have
\[
        M_a^\perp
        =
        \left\{
        \binom{-(T^*-\overline a I)^{-1}D_Ty}{y}:y\in H
        \right\}.
\]
Thus $M_a^\perp$ is the graph of the bounded operator
\[
        -(T^*-\overline a I)^{-1}D_T:H\to H.
\]
Let
\[
        \binom{x_n}{0}+
        \binom{-(T^*-\overline a I)^{-1}D_Ty_n}{y_n}
        \longrightarrow
        \binom{x}{y}
        \quad \text{in } H\oplus H.
\]
From the second component, $y_n\to y$. Since
$-(T^*-\overline a I)^{-1}D_T$ is bounded, we get
\[
        -(T^*-\overline a I)^{-1}D_Ty_n
        \longrightarrow
        -(T^*-\overline a I)^{-1}D_Ty.
\]
Therefore
\[
        x_n\to x+(T^*-\overline a I)^{-1}D_Ty.
\]
Hence
\[
\binom{x}{y}
=
\binom{x+(T^*-\overline a I)^{-1}D_Ty}{0}
+
\binom{-(T^*-\overline a I)^{-1}D_Ty}{y}
\in \mathcal D(U_a).
\]
Thus $\mathcal D(U_a)$ is closed.\\
\textbf{Step 4.}
Since the range of an isometric operator is closed, it suffices to show that the range of $U_a$ is dense in $\mathcal{H}$. Observe that
\[
(V_0(\ker V)^{\perp})^{\perp}
= (\rng V_0)^{\perp}
= (\rng V)^{\perp}
= \ker V^*
= \left\{
\begin{bmatrix}
D_{T^*} g \\
- T^* g
\end{bmatrix}
: g \in H
\right\},
\]
and
\[
((V_0-aI)(\ker V)^{\perp})^{\perp\perp}
= (V_0-aI)(\ker V)^{\perp}
=
\left\{
\begin{bmatrix}
(T-aI)h \\
D_T h
\end{bmatrix}
: h \in H
\right\}.
\]
Suppose, if possible, that
\[
\left\{
\begin{bmatrix}
D_{T^*} g \\
- T^* g
\end{bmatrix}
: g \in H
\right\}
\cap
\left\{
\begin{bmatrix}
(T-aI)h \\
D_T h
\end{bmatrix}
: h \in H
\right\}
\neq \{0\}.
\]
Then there exist $g,h \in H$ such that
\[
D_{T^*} g = (T-aI)h, \qquad -T^* g = D_T h.
\]
Hence
\begin{eqnarray*}
&& T^* D_{T^*} g = T^*(T-aI)h \\
&\Rightarrow& D_T T^* g = T^* T h - a T^* h \\
&\Rightarrow& - D_T^2 h = T^* T h - a T^* h \\
&\Rightarrow& (I-aT^*)h = 0.
\end{eqnarray*}
Since $a \in \rho(T)$, it follows that $h=0$.
Consequently,
\begin{eqnarray*}
&& (V_0(\ker V)^{\perp})^{\perp} 
   \cap ((V_0-aI)(\ker V)^{\perp})^{\perp\perp}
   = \{0\} \\
&\Rightarrow&
\big(V_0(\ker V)^{\perp}
+ ((V_0-aI)(\ker V)^{\perp})^{\perp}\big)^{\perp}
= \{0\} \\
&\Rightarrow&
\mathcal{H}
= \overline{V_0(\ker V)^{\perp}
+ ((V_0-aI)(\ker V)^{\perp})^{\perp}}
= \overline{\rng(U_a)} .
\end{eqnarray*}
\textbf{Step 5.}
For $z \in \mathbb{C}\setminus \mathbb{T}$, the operator $(U_a-zI)$ is invertible. Hence, for any such $z$ and for every $h \in \mathcal{H}$, there exist unique elements 
$f \in (\ker V)^{\perp}$ and 
$\phi \in \big((V_0-aI)(\ker V)^{\perp}\big)^{\perp}$ 
such that
\begin{eqnarray*}
h
&=& (U_a-zI)(f+\phi) \\
&=& V_0 f + a\phi - zf - z\phi \\
&=& (V_0-zI)f + (a-z)\phi .
\end{eqnarray*}
Therefore,
\[
\mathcal{H}
=
(V_0-zI)(\ker V)^{\perp}
+
\big((V_0-aI)(\ker V)^{\perp}\big)^{\perp}.
\]
Since $f$ and $\phi$ are uniquely determined, the above decomposition is unique. Hence,
\[
\mathcal{H}
=
(V_0-zI)(\ker V)^{\perp}
\;\dotplus\;
\big((V_0-aI)(\ker V)^{\perp}\big)^{\perp}.
\]
The other direct sum decomposition follows directly from the definitions.	
\end{proof}

\begin{rmk} Clearly, the space $\big((V_0-aI)(\ker V)^{\perp}\big)^{\perp}$ is a closed subspace of $\mathcal{H}$. 
Since $a$ is a point of regular type of $V_0$, there exists $\varepsilon>0$ such that the ball $B_{\varepsilon}(a)$ consists entirely of points of regular type of $V_0$. 
Since $V_0$ is an isometry, it follows that 
\[
(\mathbb{C}\setminus\mathbb{T})\cup C_{\varepsilon}(a)
\]
is a connected subset of the field of regularity of $V_0$.
Observe that
\[
(V_0(\ker V)^{\perp})^{\perp}
=
\left\{
\begin{bmatrix}
D_{T^*}g\\
-\,T^*g
\end{bmatrix}
: g\in H
\right\},
\]
which is the range of the isometric operator
\[
\begin{bmatrix}
D_{T^*} & 0\\
-\,T^* & 0
\end{bmatrix}
\]
acting on $H\oplus\{0\}$. Hence $(V_0(\ker V)^{\perp})^{\perp}$ is infinite dimensional. 
It now follows from \cite[Section~78]{Akhiezer} that $\big((V_0-aI)(\ker V)^{\perp}\big)^{\perp}$ is also an infinite dimensional subspace.
\end{rmk}
It is shown in Theorem~\ref{Theorem D} that the decomposition \eqref{D} holds for all 
$z \in \mathbb{C}\setminus\mathbb{T}$. Trivially, it holds for the point $z=a$. In the following theorem, we show that it also holds for all 
$z \in C_{\varepsilon}(a)$ for some $\varepsilon>0$. 
\begin{thm}\label{Theorem 3.5}
Let the operators $T$ and $V_0$ be as defined above, and suppose that 
$a \in \mathbb{T}$ belongs to the resolvent set $\rho(T)$. Then the set of all 
points $z \in \mathbb{C}$ for which $M_z$ is closed and the direct sum 
decomposition \eqref{D} holds is an open set.
\end{thm}

To prove the above theorem, we recall the following definitions and results 
from \cite{KKM, GM1, GM2}.

The gap between two linear manifolds $N_1$ and $N_2$ of a Hilbert space $H$, 
denoted by $\delta(N_1,N_2)$, is defined by
\[
\delta(N_1,N_2)
=
\max \left\{
\sup_{\substack{x\in N_2 \\ \|x\|=1}} \rho(x,N_1),
\;
\sup_{\substack{y\in N_1 \\ \|y\|=1}} \rho(y,N_2)
\right\},
\]
where
\[
\rho(x,N_1)=\inf_{y\in N_1}\|x-y\|.
\]

We also define $\tilde{\delta}(N_1,N_2)$ by
\[
\tilde{\delta}(N_1,N_2)
=
\max \left\{
\sup_{x\in S(N_2)} \rho(x,S(N_1)),
\;
\sup_{y\in S(N_1)} \rho(y,S(N_2))
\right\},
\]
where $S(N_1)$ denotes the unit sphere of the subspace $N_1$.

Another equivalent definition of the gap (see \cite[Section 34]{Akhiezer}) is
\[
\delta(N_1,N_2)
=
\max \left\{
\sup_{\substack{f\in \overline{N_2} \\ \|f\|=1}} 
\|(I-P_1)f\|,
\;
\sup_{\substack{g\in \overline{N_1} \\ \|g\|=1}} 
\|(I-P_2)g\|
\right\},
\]
where $P_1$ and $P_2$ denote the orthogonal projections onto 
$\overline{N_1}$ and $\overline{N_2}$, respectively.

It follows from the definitions that
\[
\delta(N_1,N_2)
=
\delta(N_2,N_1)
=
\delta(N_1^{\perp},N_2^{\perp}).
\]

The relation between $\delta(N_1,N_2)$ and $\tilde{\delta}(N_1,N_2)$ is given by
\[
\delta(N_1,N_2)
\leq
\tilde{\delta}(N_1,N_2)
\leq
2\delta(N_1,N_2).
\]

The minimal angle $\alpha(L,N)$, where $0\leq \alpha(L,N)\leq \pi/2$, between 
subspaces $L$ and $N$ of a Hilbert space $H$ is defined by
\[
\cos \alpha(L,N)
:=
\sup_{\substack{x\in L,\; y\in N \\ \|x\|=\|y\|=1}}
|\langle x,y\rangle|.
\]

\begin{lemma}\label{Lemma 1}
Let $L_1$ and $L_2$ be subspaces of $H$. Then $L_1 \cap L_2=\{0\}$ and 
$L_1 \dotplus L_2$ is closed if and only if the minimal angle between them is positive.
\end{lemma}

\begin{lemma}\label{Lemma 2}
For any three subspaces $N_1$, $N_2$, and $N_3$ of $H$, the following inequality holds:
\[
\sin\alpha(N_1,N_3)
\geq
\sin\alpha(N_1,N_2)
-
\tilde{\delta}(N_2,N_3).
\]
\end{lemma}

Now we prove Theorem \ref{Theorem 3.5}. 
\begin{proof}[Proof of Theorem \ref{Theorem 3.5}]
Let
\[
M_z = (V_0 - zI)(\ker V)^{\perp}.
\]
From Theorem \ref{Theorem D} and Remark \ref{Remark 3.2}, we know that for all 
$z \in \mathbb{C}\setminus \mathbb{T}$, the space $M_z$ is closed and the following 
direct sum decomposition holds:
\[
\mathcal{H}= M_z \dotplus M_a^{\perp}.
\]

Since $a$ is a point of regular type of $V_0$, there exists a constant $c_a>0$ 
(depending on $a$) such that
\[
\Vert (V_0- aI)f \Vert \geq c_a \Vert f \Vert.
\]

Therefore, the space $M_a$ is closed and the decomposition holds trivially for $z=a$. 
Since the set of points of regular type is open, there exists $\varepsilon_1>0$ such that 
every $z \in B_{\varepsilon_1}(a)$ is a point of regular type of $V_0$. Consequently, 
for such $z$, the space $M_z$ is closed.

Let $\varepsilon_2=\frac{c_a}{3}$. We will show that for all $z \in C_{\varepsilon_2}(a)$, 
the direct sum decomposition holds. Let $z \in C_{\varepsilon_2}(a)$. Then for all 
$h \in M_z^{\perp}$ with $\Vert h \Vert =1$, we have
\begin{eqnarray}
\Vert(I-P_a)h \Vert 
&=& \sup_{f \in H \oplus \{ 0 \}} 
\frac{\vert\langle h, (V_0- aI)f \rangle \vert}
{\Vert (V_0- aI)f \Vert} \nonumber \\
&=& \sup_{f \in H \oplus \{ 0 \}} 
\frac{\vert\langle h, (V_0- zI)f+(z-a)f \rangle \vert}
{\Vert (V_0- aI)f \Vert} \nonumber \\
&=& \sup_{f \in H \oplus \{ 0 \}} 
\frac{\vert\langle h, (z-a)f \rangle \vert}
{\Vert (V_0- aI)f \Vert} \label{eq3.4}\\
&\leq & \sup_{f \in H \oplus \{ 0 \}} 
\frac{c_a}{3}
\frac{\vert\langle h,f \rangle \vert}
{c_a\Vert f \Vert} 
\leq \frac{1}{3}. \nonumber
\end{eqnarray}

Similarly, for all $h \in M_a^{\perp}$ with $\Vert h \Vert =1$, we obtain
\begin{equation}\label{eq3.5} 
\Vert(I-P_z)h \Vert \leq \frac{1}{3}.
\end{equation}
Now, by the definition of the gap and by inequalities \eqref{eq3.4} and \eqref{eq3.5}, we obtain
\[
\delta(M_a^{\perp}, M_z^{\perp}) \leq \frac{1}{3}.
\]
Hence,
\[
\tilde{\delta}(M_a, M_z)\leq \frac{2}{3}.
\]
Observe that
\[
\sup_{\substack{x \in M_a,\; y \in M_a^{\perp} \\
\Vert x\Vert =\Vert y\Vert=1}} 
\vert \langle x,y \rangle \vert= 0.
\]
This implies that $
\sin \alpha(M_a, M_a^{\perp})=1.
$
Similarly,
$
\sin \alpha(M_z, M_z^{\perp})=1
$
for all $z \in C_{\varepsilon_2}(a)$.
By Lemma \ref{Lemma 2}, we have
\begin{eqnarray*}
\sin \alpha(M_z, M_a^{\perp})
&\geq&  \sin \alpha(M_a, M_a^{\perp})- \tilde{\delta}(M_z, M_a)\\
&=&1- \tilde{\delta}(M_z, M_a)>0.
\end{eqnarray*}
Hence,
\[
\alpha(M_z, M_a^{\perp}) >0.
\]
Therefore, by Lemma \ref{Lemma 1}, we obtain
\begin{equation}\label{3.4} 
M_z \cap M_a^{\perp}= \{ 0\}
\quad \text{and}\quad 
M_z \dotplus M_a^{\perp}~\text{is closed}.
\end{equation}
Interchanging $z$ and $a$ in the above arguments, and using the symmetry
\[
\tilde{\delta}(M_z, M_a)= \tilde{\delta}(M_a, M_z),
\]
we also obtain
\begin{equation}\label{3.5} 
M_a \cap M_z^{\perp}= \{ 0\}
\quad\text{and}\quad 
M_a \dotplus M_z^{\perp}~\text{is closed}.
\end{equation}
We recall that for any two closed subspaces $M$ and $N$ of a Banach space,
\[
(M+N)^{\perp}= M^{\perp} \cap N^{\perp}.
\]
Now, from \eqref{3.4} and \eqref{3.5}, we obtain
$
\mathcal{H}= \{ 0\}^{\perp} 
= (M_a \cap M_z^{\perp})^{\perp}
= \overline{M_z +M_a^{\perp}}
= M_z +M_a^{\perp}.
$
Hence, for all $z \in C_{\varepsilon}(a)$, where $
\varepsilon= \min \{ \varepsilon_1, \varepsilon_2 \},$
the space $M_z$ is closed and the direct sum decomposition \eqref{D} holds.
\end{proof}



Throughout the rest of the paper, we will consider the following domain $\Omega_a$ which is symmetric with respect to the unit circle and contains the open unit disc: 
\begin{equation}\label{Omega}
\Omega_a= (\mathbb{C}\setminus\mathbb{T}) \cup C_{\varepsilon}(a) 
\end{equation}


Note that for $z \in \mathbb{T}$ the following implications hold. These follow directly from Remarks \ref{Remark 3.2} and \ref{Remark 3.3}:
\[
\begin{aligned}
z \in \Omega_a
&\Rightarrow z \text{ is a point of regular type of } V \\
&\Leftrightarrow z \in \rho(V) \\
&\Leftrightarrow z \in \rho(T).
\end{aligned}
\]

\section{Abstract construction of the vector valued RKHS and functional model}
\label{Section 4}
In this section, we construct a vector valued RKHS from the direct sum decomposition \eqref{D} given in the previous section, and show that under some assumptions, it is a vector valued de Branges space.

Denote\[Y := \big((V_0-aI)(\ker V)^{\perp}\big)^{\perp}.\] By \eqref{D}, for every $f \in \mathcal{H}$ there exists a unique 
$g \in (\ker V)^{\perp}$ (depending on $z$) and a projection operator 
$P_Y(z)$ such that
\begin{equation}\label{4.1}
f = (V_0 - zI)g + P_Y(z)f, 
\qquad z \in \Omega_a.
\end{equation}
The operator $P_Y(z)$ is a bounded linear operator for all $z \in \Omega_a$. 
Fix $f \in \mathcal{H}$ and consider the map from $\Omega_a$ to $Y$ defined by
\[
z \mapsto P_Y(z)f.
\]
We denote this $Y$-valued function by $f_Y$, where
\[
f_Y(z) = P_Y(z)f.
\]
Define the space $\mathbf{H}$ by
\[
\mathbf{H} = \{ f_Y : f \in \mathcal{H} \}.
\]
Then $\mathbf{H}$ consists of vector valued analytic functions defined on $\Omega_a$. Consider the map  
$\Psi : \mathcal{H} \to \mathbf{H}$ defined by
\begin{equation}\label{Psi}
\Psi(f) = f_Y.
\end{equation}

Suppose that $\Psi f=0$ for some $f \in \mathcal{H}$. Then, for all $z \in \Omega_a$, we have
\[
0 = f_Y(z) = P_Y(z)f.
\]
This implies that
\begin{eqnarray*}
f &\in& (V_0-zI)(\ker V)^{\perp} = (V-zI)(\ker V)^{\perp} 
\quad \text{for all } z \in \Omega_a,\\
&\Rightarrow& f \in \bigcap_{z\in \Omega_a} (V-zI)(\ker V)^{\perp}.
\end{eqnarray*}

To prove that the map $\Psi$ is injective, we first state and prove the following general result for cnu partial isometries.

\begin{lemma}
Let $V$ be a cnu partial isometry on a Hilbert space $\mathcal{H}$. Then
\[
\bigcap_{z\in \mathbb{C}\setminus \mathbb{T}} (V-zI)(\ker V)^{\perp}=\{0\}.
\]
\end{lemma}
\begin{proof}
Let if possible there exists a non zero $h$ such that $$h \in \bigcap_{z\in \mathbb{C}\setminus \mathbb{T}} (V-zI)(\ker V)^{\perp}.$$
First, let \(|z|>1\). Since \(\|V\|\leq 1\), the operator \(V-zI\) is invertible and
\[
(V-zI)^{-1}
=
-\sum_{n=0}^{\infty} z^{-n-1}V^n.
\]
Since, \(h\in (V-zI)(\ker V)^{\perp}\), there exists $x_z \in (\ker V)^{\perp}$ such that 
\begin{eqnarray*}
&& h=(V-zI)x_z\\
&\Rightarrow & x_z=(V-zI)^{-1}h.
\end{eqnarray*} Hence
\[
0=\mathbf{P}_{\ker V}x_z
=
-\sum_{n=0}^{\infty} z^{-n-1}\mathbf{P}_{\ker V}V^n h,
\qquad |z|>1,
 \]
where $\mathbf{P}_{\ker V}$ denotes the orthogonal projection onto $\ker V$.
Since the above vector valued analytic function vanishes identically on \(\{|z|>1\}\), all its coefficients must vanish. Therefore
\[
\mathbf{P}_{\ker V}V^n h=0,
\qquad n\geq 0,
\]
that is,
\[
V^n h\in (\ker V)^{\perp},
\qquad n\geq 0.
\]

Next, let \(|z|<1\). Since \(h\in (V-zI)(\ker V)^{\perp}\), there exists \(x_z\in (\ker V)^{\perp}\) such that
\[
h=(V-zI)x_z.
\]
Applying \(V^*\) and using that \(V^*V=I\) on \((\ker V)^{\perp}\), we get
\[
V^*h=V^*Vx_z-zV^*x_z=x_z-zV^*x_z,
\]
so that
\[
x_z=(I-zV^*)^{-1}V^*h=\sum_{n=0}^{\infty} z^n V^{*\,n+1}h.
\]
Applying \(\mathbf{P}_{\ker V^*}\), the orthogonal projection onto $\ker V^*$, to the identity \(h=(V-zI)x_z\), we obtain
\[
\mathbf{P}_{\ker V^*} h=\mathbf{P}_{\ker V^*}(Vx_z-z x_z).
\]
Since, $Vx_z \in V(\ker V)^{\perp}= \rng V= (\ker V^*)^{\perp}$, we get that $\mathbf{P}_{\ker V^*}Vx_z=0$. Hence, $$\mathbf{P}_{\ker V^*} h=-z\,\mathbf{P}_{\ker V^*} x_z$$
Substituting the above expansion of \(x_z\), we get
\begin{eqnarray*}
\mathbf{P}_{\ker V^*} h
&=&
-z\,\sum_{n=0}^{\infty} z^n \mathbf{P}_{\ker V^*} V^{*\,n+1}h\\
&=&- \sum_{n=0}^{\infty} z^{n+1} \mathbf{P}_{\ker V^*} V^{*\,n+1}h
\end{eqnarray*}

The right-hand side has no constant term, hence \(\mathbf{P}_{\ker V^*} h=0\). It follows that all coefficients vanish, and thus
\[
\mathbf{P}_{\ker V^*} V^{*\,n}h=0,
\qquad n\geq 0.
\]
Therefore
\[
V^{*\,n}h \in (\ker V^*)^{\perp}=\rng V,
\qquad n\geq 0.
\]

Now define
\[
\mathcal{N}
:=
\overline{\operatorname{span}}
\bigl(
\{V^n h:n\geq 0\}\cup \{V^{*n}h:n\geq 1\}
\bigr).
\]
We claim that \(\mathcal{N}\) is reducing for \(V\). Indeed, for \(n\geq 0\),
\[
V(V^n h)=V^{n+1}h\in \mathcal{N},
\]
and for \(m\geq 1\), since \( V^{*m}h \in \rng V \), we have \(VV^*=I\) on these vectors, so
\[
V(V^{*m}h)=VV^*V^{*(m-1)}h=V^{*(m-1)}h\in \mathcal{N}.
\]
Thus \(V\mathcal{N}\subseteq \mathcal{N}\).

Similarly, for \(m\geq 1\),
\[
V^*(V^{*m}h)=V^{*(m+1)}h\in \mathcal{N},
\]
and for \(n\geq 1\), since \(V^n h\in (\ker V)^{\perp}\), we have \(V^*V=I\) on these vectors, so
\[
V^*(V^n h)=V^{n-1}h\in \mathcal{N}.
\]
Also \(V^*h\in \mathcal{N}\). Hence \(V^*\mathcal{N}\subseteq \mathcal{N}\), and therefore \(\mathcal{N}\) is reducing for \(V\).

Finally, we show that \(V|_{\mathcal{N}}\) is unitary. Using the same reasoning as above, we get that for $m \geq 1$, $$VV^*(V^{*m}h)= V^{*m}h,$$ and $$V^*V(V^{*m}h)= V^*VV^*V^{*(m-1)}h= V^*V^{*(m-1)}h= V^{*m}h.$$
Similarly, for $n\geq 1$, $$VV^*(V^nh)= VV^*V(V^{n-1}h)= V(V^{n-1}h)= V^nh,$$
and for $n \geq 0$, $$V^*V(V^nh)=V^nh.$$
By linearity and continuity, we get that $\mathcal{N}$ is a non-zero proper closed subspace reducing $V$ and $V \big|_{\mathcal{N}}$ is unitary. This contradicts the fact that $V$ is cnu. Hence, $$\bigcap_{z\in \mathbb{C}\setminus \mathbb{T}} (V-zI)(\ker V)^{\perp}= \{0 \}.$$
\end{proof}
Since
\[
\bigcap_{z\in \Omega_a} (V-zI)(\ker V)^{\perp}
\subseteq 
\bigcap_{z\in \mathbb{C}\setminus \mathbb{T}} (V-zI)(\ker V)^{\perp},
\]
it follows from the above lemma that $f=0$. Hence, the map $\Psi$ is injective.

Therefore, $\mathbf{H}$ is a vector space of analytic vector valued functions with the usual pointwise addition and scalar multiplication. We define an inner product on $\mathbf{H}$ by
\[
\langle f_Y , g_Y \rangle_{\mathbf{H}}
:= \langle f , g \rangle_{\mathcal{H}}.
\]

With this inner product, the map $\Psi$ becomes a unitary operator from $\mathcal{H}$ onto $\mathbf{H}$. Consequently, $\mathbf{H}$ is a Hilbert space. Moreover, it is also a RKHS, since the pointwise evaluations are bounded. 
Indeed,
\[
\|f_Y(z)\| = \|P_Y(z)f\|
\leq \|P_Y(z)\|\,\|f\|_{\mathcal{H}} .
\]
Observe that for all $f \in (\ker V)^{\perp}$,
\begin{eqnarray*}
V_0 f
&=& V_0 f - zf + zf \\
&=& (V_0 - zI)f + zf \\
&=& (V_0 - zI)f + z\big((V_0 - zI)g + P_Y(z)f\big) \\
&=& (V_0 - zI)(f + zg) + zP_Y(z)f .
\end{eqnarray*}
This implies that
\[
P_Y(z)(V_0 f) = z\, P_Y(z)f .
\]
Hence the operator $V_0$ is unitarily equivalent to the multiplication 
operator $\mathfrak{T}$ on its domain $\mathcal{D}(\mathfrak{T})$ given by $$(\mathfrak{T}f)(z)= z f(z), \quad f \in \mathcal{D}(\mathfrak{T}).$$ In particular,
\begin{equation}\label{T}
T = \mathbf{P}_H V_0
   = \mathbf{P}_H V\big|_H
   \cong 
   \mathbf{P}_{\mathcal{D}(\mathfrak{T})}
   \mathfrak{T}\big|_{\mathcal{D}(\mathfrak{T})},
\end{equation}
where $\mathbf{P}_H$ denotes the orthogonal projection onto the subspace $H$.

In the following theorem of this section, we show that under some conditions the space $\mathbf{H}$ is a vector valued 
de Branges space $\mathcal{B}(\mathfrak{E})$.
\begin{thm}\label{Theorem M}
Let $T$ be a cnu contraction operator on a Hilbert space $H$ such that a point $a \in \mathbb{T}$ belongs to $\rho(T)$. 
Let $V_0$ and $V$ be the contraction operators defined in \eqref{V0} and \eqref{V}, respectively, and let $\mathbf{H}$ be the RKHS as described above. Suppose that there exists at least one $\beta \in \mathbb{D}$ such that the following conditions hold:
\begin{itemize}
\item[i)] $\dim(M_z^{\perp} \cap M_{\beta}) < \infty$ and $M_z^{\perp} + M_{\beta}$ is closed for all $z \in \Omega_a\setminus \mathbb{T}$,
\item[ii)] $\dim(M_z^{\perp} \cap M_{\frac{1}{\overline{\beta}}}) < \infty$ and $M_z^{\perp} + M_{\frac{1}{\overline{\beta}}}$ is closed for all $z \in \Omega_a\setminus \mathbb{T}$,
\item[iii)] $\dim(M_0 \cap M_{\beta}^{\perp}) < \infty$ and $\dim(M_0 \cap M_{\frac{1}{\overline{\beta}}}^{\perp}) < \infty$,
\end{itemize}
where $M_z = (V_0 - zI)(\ker V)^{\perp}$.
Then the space $\mathbf{H}$ is a vector valued de Branges space $\mathcal{B}(\mathfrak{E})$ of analytic functions on $\Omega_a$. 
Moreover, the operator $V_0$ on $H \oplus \{0\}$ is unitarily equivalent to the multiplication operator $\mathfrak{T}$ on $\mathcal{D}(\mathfrak{T}) \subseteq \mathbf{H}$. Consequently, the operator $T$ is unitarily equivalent to $\mathbf{P}_{\mathcal{D}(\mathfrak{T})}
   \mathfrak{T}\big|_{\mathcal{D}(\mathfrak{T})}$.
\end{thm}

\begin{rmk}We want to remark here that the pairs $(M_z^{\perp}, M_{\beta})$ and $(M_z^{\perp}, M_{\frac{1}{\overline{\beta}}})$ satisfying the conditions i) and ii) given in the above theorem are semi-Fredholm. Recall that a pair $(A,B)$ of closed subspaces of a Banach space is said to be semi-Fredholm if $A+B$ is closed and at least one of the $\dim(A\cap B)$ and $\operatorname{codim}(A+B)$ is finite. For a detailed treatment of semi-Fredholm pairs of subspaces, we refer the reader to \cite[Chapter~4, Section~4]{kato}.
\end{rmk}

We now give a general criterion which will be used to produce concrete examples
satisfying the hypotheses of Theorem~\ref{Theorem M}. This criterion reduces the
verification of the subspace conditions in Theorem~\ref{Theorem M} to the
Fredholmness of certain naturally associated operators.


\begin{lemma}\label{Lemma 4.4}
Let $T\in B(H)$ be a cnu contraction on an
infinite dimensional Hilbert space $H$, and suppose that there exists $a\in\rho(T)\cap\mathbb T$. 
Assume that there exists a non zero $\beta\in\mathbb D$ such that the following two conditions hold:
\begin{enumerate}
\item for each $\alpha\in\{\beta,\frac{1}{\overline{\beta}}\}$ and each
$z\in\Omega_a\setminus\mathbb T$, the operator
\[
        \Phi_{z,\alpha}:=
        I-\alpha T^*-\overline z\,T+\overline z\alpha I
\]
is Fredholm;

\item for each $\alpha\in\{\beta,\frac{1}{\overline{\beta}}\}$, the operator $I-\overline\alpha T$ is Fredholm.
\end{enumerate}
Then $T$ satisfies all the hypotheses of Theorem~\ref{Theorem M}.
\end{lemma}
\begin{proof}
For $w\in\Omega_a$, define $A_w:H\to H\oplus H$
by
\[
        A_wx=((T-wI)x,D_Tx).
\]
Then $M_w=\mathrm{rng}A_w$.

We verify conditions $(i)$ and $(ii)$ of Theorem~\ref{Theorem M}. Fix $\alpha\in\{\beta,\frac{1}{\overline{\beta}}\}$ and $z\in\Omega_a\setminus\mathbb T$. Suppose that $y\in M_z^\perp\cap M_\alpha$. Since $y\in M_\alpha$, there exists $x\in H$ such that $ y=A_\alpha x$. The condition $y\in M_z^\perp$ is equivalent to $ A_\alpha x\perp \mathrm{rng}A_z$. Hence $A_z^*A_\alpha x=0$.
Therefore
\[
        M_z^\perp\cap M_\alpha
        \subseteq
        A_\alpha\ker(A_z^*A_\alpha).
\]
In particular,
\[
        \dim(M_z^\perp\cap M_\alpha)
        \leq
        \dim\ker(A_z^*A_\alpha).
\]

Now,
\[
\begin{aligned}
        A_z^*A_\alpha
        &=
        (T-zI)^*(T-\alpha I)+D_T^2  \\
        &=
        (T^*-\overline z I)(T-\alpha I)+I-T^*T  \\
        &=
        I-\alpha T^*-\overline z\,T+\overline z\alpha I.
\end{aligned}
\]
Thus
\[
        A_z^*A_\alpha=\Phi_{z,\alpha}.
\]
By assumption, $\Phi_{z,\alpha}$ is Fredholm. Therefore, $\dim\ker(A_z^*A_\alpha)<\infty$.
It follows that $\dim(M_z^\perp\cap M_\alpha)<\infty$. Taking $\alpha=\beta$ and $\alpha=\frac{1}{\overline{\beta}}$, we get
\[
        \dim(M_z^\perp\cap M_\beta)<\infty
\]
and
\[
        \dim(M_z^\perp\cap M_{\frac{1}{\overline{\beta}}})<\infty.
\]

Next, we prove the closedness of $M_z^\perp+M_\alpha$. Let $Q_z$ denote the
orthogonal projection from $H\oplus H$ onto $M_z$. Since $M_z$ is closed, $Q_z$
is well-defined, and $\ker Q_z=M_z^\perp$.
It is easy to observe that 
\[
        M_z^\perp+M_\alpha
        =
        Q_z^{-1}\bigl(Q_z(M_\alpha)\bigr).
\]
It remains to show that $Q_z(M_\alpha)$ is closed. Put
\[
        G_z=A_z^*A_z.
\]
Since $A_z$ is bounded below, $G_z$ is boundedly invertible on $H$. Moreover,
the orthogonal projection onto $M_z=\mathrm{rng}A_z$ is given by
\[
        Q_z=A_zG_z^{-1}A_z^*.
\]
Therefore, for $x\in H$,
\[
        Q_zA_\alpha x
        =
        A_zG_z^{-1}A_z^*A_\alpha x
        =
        A_zG_z^{-1}\Phi_{z,\alpha}x.
\]
Hence
\[
        Q_z(M_\alpha)
        =
        A_zG_z^{-1}\bigl(\mathrm{rng}\Phi_{z,\alpha}\bigr).
\]
The operator
\[
        A_zG_z^{-1}:H\to M_z
\]
is a Banach space isomorphism. Since $\Phi_{z,\alpha}$ is Fredholm, its range is
closed. Therefore $Q_z(M_\alpha)$ is closed. Consequently, $M_z^\perp+M_\alpha$ is closed. Taking $\alpha=\beta$ and $\alpha=\frac{1}{\overline{\beta}}$, we obtain that
\[
        M_z^\perp+M_\beta
        \quad\text{and}\quad
        M_z^\perp+M_{\frac{1}{\overline{\beta}}}
\]
are closed. Thus conditions $(i)$ and $(ii)$ of Theorem~\ref{Theorem M} are
satisfied.

Finally, we verify condition $(iii)$. Fix $\alpha\in\{\beta,\frac{1}{\overline{\beta}}\}$. Following the same steps as before, we get that
\[
        \dim(M_0\cap M_\alpha^\perp)
        \leq
        \dim\ker(A_\alpha^*A_0).
\]
Now,
\[
\begin{aligned}
        A_\alpha^*A_0
        &=
        (T-\alpha I)^*T+D_T^2  \\
        &=
        (T^*-\overline\alpha I)T+I-T^*T  \\
        &=
        I-\overline\alpha T.
\end{aligned}
\]
By assumption, $I-\overline\alpha T$ is Fredholm. Therefore $\dim\ker(A_\alpha^*A_0)<\infty$. Consequently, $\dim(M_0\cap M_\alpha^\perp)<\infty$.
Taking $\alpha=\beta$ and $\alpha=\frac{1}{\overline{\beta}}$, we get
\[
        \dim(M_0\cap M_\beta^\perp)<\infty
\]
and
\[
        \dim(M_0\cap M_{\frac{1}{\overline{\beta}}}^\perp)<\infty.
\]
This verifies condition $(iii)$ of Theorem~\ref{Theorem M}. Hence $T$ satisfies
all the hypotheses of Theorem~\ref{Theorem M}.
\end{proof}

\begin{rmk}\label{Remark 4.5}In particular, Lemma~\ref{Lemma 4.4} applies whenever $T-I$ is compact, because in that case the operators $\Phi_{z,\alpha}$ and $I-\overline\alpha T$ are compact perturbations of non-zero scalar multiples of the identity whenever $z\neq 1$ and $\alpha\neq 1$.
\end{rmk}
\begin{rmk}\label{Remark 4.6}More generally, Lemma~\ref{Lemma 4.4} also applies if \(T\) is essentially unitary, that is, if
\[
I-T^*T \quad \text{and} \quad I-TT^*
\]
are compact. Indeed, let $\pi: B( H)\to
         B( H)/ K( H)$ be the quotient map onto the Calkin algebra. Since compact operators vanish in the Calkin algebra, essential unitarity implies
\[
        \pi(T)^{*}\pi(T)=I
        \quad\text{and}\quad
        \pi(T)\pi(T)^{*}=I.
\]
Thus \(\pi(T)\) is a unitary element of the Calkin algebra. Since \(\alpha\notin\mathbb T\) and \(z\notin\mathbb T\), the elements
\[
I-\overline{\alpha}\pi(T)
\quad\text{and}\quad
(I-\overline z\,\pi(T))(I-\alpha\pi(T)^*)
\]
are invertible in the Calkin algebra. Hence, by \cite[Chapter XI, Theorem 5.2]{GGK}, the operators \(I-\overline{\alpha}T\) and \(\Phi_{z,\alpha}\) are Fredholm. For more details about essentially unitary operators and the properties used here, see \cite[Chapter 1-2]{CM}. In particular, if $T$ is a cnu contraction such that $\rho(T) \cap \mathbb{T} \neq \emptyset$, and both the defect operators $D_T$ and $D_T^*$ are compact, then the hypothesis of Theorem \ref{Theorem M} are satisfied.
\end{rmk}

The following examples illustrate different classes of cnu contractions satisfying the standing assumptions of Theorem~\ref{Theorem M}. The first example is a diagonal essentially unitary contraction, the second is a non-normal block weighted-shift-type contraction with \(T-I\) compact, and the last four arise from scalar and vector valued model spaces. In each case, Lemma \ref{Lemma 4.4}, together with Remark \ref{Remark 4.5} or Remark \ref{Remark 4.6}, implies that all the hypotheses of
Theorem~\ref{Theorem M} are satisfied.

\begin{example}
Let \(H=\ell^2(\mathbb N)\), and let \(\{e_n\}_{n\geq 1}\) be the standard orthonormal basis of \(H\). Let \((\lambda_n)_{n\geq 1}\subset \mathbb D\) be a
sequence such that $|\lambda_n|<1$ for every $n$, $|\lambda_n|\to 1$.
Assume that
\[
        \overline{\{\lambda_n:n\in\mathbb N\}}\neq \mathbb T.
\]
Choose and fix a point
\[
        a\in \mathbb T\setminus \overline{\{\lambda_n:n\in\mathbb N\}}.
\]
Define \(T\in B(H)\) by
\[
        T e_n=\lambda_n e_n,\qquad n\geq 1.
\]
Then \(T\) is a diagonal contraction. Indeed, if
\(x=\sum_{n=1}^{\infty}x_ne_n\in H\), then
\[
        \|Tx\|^2
        =
        \sum_{n=1}^{\infty}|\lambda_n|^2 |x_n|^2
        \leq
        \sum_{n=1}^{\infty}|x_n|^2
        =
        \|x\|^2.
\]
Moreover, \(T\) is cnu. In fact, for every non-zero
\(x=\sum_{n=1}^{\infty}x_ne_n\in H\), we have
$\Vert Tx \Vert < \Vert x \Vert$
because \(|\lambda_n|<1\) for every \(n\). Hence \(T\) cannot be unitary on any
non-zero reducing subspace of \(H\). Since \(T\) is a diagonal normal operator, its spectrum is
\[
        \sigma(T)=\overline{\{\lambda_n:n\in\mathbb N\}}.
\]
By assumption, \(a\notin \overline{\{\lambda_n:n\in\mathbb N\}}\). Therefore $a\in \rho(T)\cap \mathbb T$.
Finally, we show that \(T\) is essentially unitary. Since
\[
        T^*T e_n = |\lambda_n|^2 e_n
        \quad\text{and}\quad
        TT^* e_n = |\lambda_n|^2 e_n,
\]
we get
\[
        (I-T^*T)e_n = (1-|\lambda_n|^2)e_n
        \quad\text{and}\quad
        (I-TT^*)e_n = (1-|\lambda_n|^2)e_n.
\]
Since \(|\lambda_n|\to 1\), we have \(1-|\lambda_n|^2\to 0\). Hence both
\(I-T^*T\) and \(I-TT^*\) are compact diagonal operators. Thus \(T\) is
essentially unitary.

Consequently, by Remark \ref{Remark 4.6} and Lemma \ref{Lemma 4.4}, \(T\) satisfies all the
hypotheses of Theorem \ref{Theorem M}.
\end{example}

\begin{example}
Let $H=\bigoplus_{n=1}^{\infty}H_n$, where $ 
        H_n=\operatorname{span}\{e_n,f_n\}\cong \mathbb C^2$.
For each $n\geq 1$, put
\[
        r_n=\frac{n+1}{n+2},
        \qquad
        \gamma_n=\frac{1-r_n^2}{2}.
\]
Then $0<r_n<1$, $r_n\to 1$, and $\gamma_n\to 0$. Define
$T_n:H_n\to H_n$ by
\[
        T_ne_n=r_ne_n+\gamma_n f_n,
        \qquad
        T_nf_n=r_nf_n.
\]
Equivalently, with respect to the ordered basis $\{e_n,f_n\}$ of $H_n$,
\[
        T_n=
        \begin{pmatrix}
        r_n & 0\\
        \gamma_n & r_n
        \end{pmatrix}.
\]
Thus $T_n$ is a non-normal weighted-shift-type Jordan block. Define $T=\bigoplus_{n=1}^{\infty}T_n$ on $H$.
We first note that $T$ is a contraction. A direct computation gives
\[
        I-T_n^*T_n
        =
        \begin{pmatrix}
        1-r_n^2-\gamma_n^2 & -r_n\gamma_n\\
        -r_n\gamma_n & 1-r_n^2
        \end{pmatrix}.
\]
Putting $a_n=1-r_n^2$, we have $\gamma_n=a_n/2$. Hence
\[
        1-r_n^2-\gamma_n^2
        =
        a_n-\frac{a_n^2}{4}>0
\]
and
\[
        \det(I-T_n^*T_n)
        =
        \frac34a_n^2>0.
\]
Thus $I-T_n^*T_n>0$ for every $n$. Therefore each $T_n$ is a strict contraction,
and consequently $T$ is a contraction. Moreover, $T$ is cnu. Indeed, if
$x=\bigoplus_{n=1}^{\infty}x_n\in H$ is non-zero, then
\[
        \|x\|^2-\|Tx\|^2
        =
        \sum_{n=1}^{\infty}
        \bigl(\|x_n\|^2-\|T_nx_n\|^2\bigr)>0.
\]
Hence $\|Tx\|<\|x\|$ for every non-zero $x\in H$, and so $T$ cannot be unitary
on any non-zero reducing subspace.

Next, since each $T_n$ has spectrum $\{r_n\}$ and $r_n\to1$, we have
\[
        \sigma(T)=\{r_n:n\in\mathbb N\}\cup\{1\}.
\]
Thus $\mathbb T\setminus\{1\}\subset\rho(T)$.
In particular, we may choose $a=-1\in\rho(T)\cap\mathbb T$.
Finally,
\[
        T_n-I_{H_n}
        =
        \begin{pmatrix}
        r_n-1 & 0\\
        \gamma_n & r_n-1
        \end{pmatrix}.
\]
Since $r_n\to1$ and $\gamma_n\to0$, we have $\|T_n-I_{H_n}\|\to0$.
Therefore
\[
        T-I=\bigoplus_{n=1}^{\infty}(T_n-I_{H_n})
\]
is compact. Consequently, by Remark \ref{Remark 4.5} and Lemma \ref{Lemma 4.4}, $T$ satisfies all the hypotheses of Theorem~\ref{Theorem M}.
\end{example}

\begin{example}[Scalar valued model spaces]\label{Example 4.9}
Let $\theta$ be an inner function such that the corresponding model space
\[
        K_\theta := H^2 \ominus \theta H^2
\]
is infinite dimensional. Equivalently, \(\theta\) is not a finite Blaschke product (see \cite[Proposition 5.19]{GMR}). Suppose further that there exists a non-empty open arc $\Gamma \subset \mathbb T$ such that $\theta$ admits an analytic continuation through a neighbourhood of $\Gamma$. Choose and fix a point $a \in \Gamma$.
Let
\[
        T=S_\theta:=\mathbf{P}_{K_\theta}S|_{K_\theta},
\]
where $S$ denotes the unilateral shift on $H^2$ and $\mathbf{P}_{K_\theta}$ denotes the orthogonal projection from $H^2$ onto $K_\theta$. Then $S_\theta$ is the compressed shift on the scalar valued model space $K_\theta$ (see \cite[Section 9.2]{GMR} for more details).

We claim that $T=S_\theta$ satisfies all the hypotheses of Theorem \ref{Theorem M}. First, $T$ is a contraction, being a compression of the unilateral shift. Moreover, $T$ is cnu. Indeed, by \cite[Corollary 9.16]{GMR}, the compressed shift $S_\theta$ is irreducible; that is, it has no proper non-trivial reducing subspaces. On the other hand, by \cite[Lemma 9.9]{GMR}, the following identities hold:
\begin{equation}\label{id}
I-S_\theta S_\theta^*
        =
        k_0^\theta \otimes k_0^\theta \quad \text{and} \quad  I-S_\theta^*S_\theta
        =
        S^*\theta \otimes S^*\theta
\end{equation}
where $k_0^\theta$ is the reproducing kernel of $K_\theta$ at $0$. Hence $S_\theta$ is not unitary on $K_\theta$. Since the only reducing subspaces of $S_\theta$ are $\{0\}$ and
$K_\theta$, it follows that $S_\theta$ has no non-zero reducing subspace on which
it is unitary. Thus $S_\theta$ is cnu.

Since $\theta$ admits analytic continuation through the arc
$\Gamma$, the points of $\Gamma$ do not belong to the spectrum $\sigma(\theta)$
of the inner function $\theta$; see \cite[Proposition 7.20]{GMR}. Moreover, by the
Liv\v{s}ic--M\"oller theorem (see \cite[Theorem 9.22]{GMR}),
\[
        \sigma(S_\theta)=\sigma(\theta).
\] Hence
\[
        a\notin \sigma(S_\theta).
\]
Therefore
\[
        a\in \rho(S_\theta)\cap \mathbb T .
\]

Next, by the identities \eqref{id}, both
\[
        I-S_\theta^*S_\theta
        \quad \text{and} \quad
        I-S_\theta S_\theta^*
\]
are rank-one operators. In particular, they are compact. Therefore $S_\theta$ is
essentially unitary. Hence, by Remark \ref{Remark 4.6} and Lemma \ref{Lemma 4.4}, $T$ satisfies all the hypotheses of Theorem~\ref{Theorem M}.
\end{example}

\begin{example}[Vector valued model spaces with matrix valued inner functions]
Let $\mathcal E$ and $\mathcal E_*$ be finite dimensional Hilbert spaces, and let
\[
        \Theta:\mathbb D\longrightarrow  B(\mathcal E,\mathcal E_*)
\]
be a purely contractive matrix valued inner function. Assume that the corresponding vector valued model space
\[
        K_\Theta
        :=
        H^2(\mathcal E_*)\ominus \Theta H^2(\mathcal E)
\]
is infinite dimensional. Suppose further that there exists a non-empty open arc $\Gamma\subset \mathbb T$ such that $\Theta$ admits an analytic continuation to a neighbourhood of $\Gamma$, and
\[
        \Theta(\zeta):\mathcal E\longrightarrow \mathcal E_*
\]
is unitary for every $\zeta\in\Gamma$. Choose and fix a point $a\in\Gamma$.
Let
\[
        T=S_\Theta:=\mathbf{P}_{K_\Theta}M_z|_{K_\Theta},
\]
where $M_z$ denotes the unilateral shift on $H^2(\mathcal E_*)$ and $\mathbf{P}_{K_\Theta}$ denotes the orthogonal projection from $H^2(\mathcal E_*)$ onto $K_\Theta$. Then $S_\Theta$ is the compressed shift on the vector valued model space $K_\Theta$; see \cite[Chapter VI, Section 3]{Nagy-Foais}.

We claim that $T=S_\Theta$ satisfies all the hypotheses of Theorem \ref{Theorem M}. First, $T$ is a contraction, being the compression of the
unilateral shift. Moreover, $T$ is cnu. Indeed, since
$\Theta H^2(\mathcal E)$ is invariant under $M_z$, the space $K_\Theta$ is
invariant under $M_z^*$, and hence
\[
        S_\Theta^*=M_z^*|_{K_\Theta}.
\]
As $(M_z^*)^n\to 0$ strongly on $H^2(\mathcal E_*)$, it follows that
\[
        (S_\Theta^*)^n f\to 0
        \qquad (f\in K_\Theta).
\]
Thus $S_\Theta$ cannot have a non-zero reducing subspace on which it is
unitary. Hence $S_\Theta$ is a cnu contraction.

Since $\Theta$ is purely contractive, the Sz.-Nagy-Foias characteristic function of $S_\Theta$ coincides with $\Theta$; see \cite[Chapter VI, Section 3, Theorem 3.1]{Nagy-Foais}. By the assumption that
$\Theta$ admits analytic continuation through $\Gamma$ and is unitary on
$\Gamma$, \cite[Chapter VI, Section 4, Theorem 4.1]{Nagy-Foais} gives
\[
        a\in \rho(S_\Theta)\cap\mathbb T .
\]

It remains to verify the additional Fredholm-type hypotheses appearing in
Theorem \ref{Theorem M}. Since $\mathcal E$ and $\mathcal E_*$ are finite
dimensional, the defect spaces of the model operator $S_\Theta$ are finite
dimensional. Equivalently,
\[
        I-S_\Theta^*S_\Theta
        \quad\text{and}\quad
        I-S_\Theta S_\Theta^*
\]
are finite rank operators. In particular, they are compact. Hence $S_\Theta$ is essentially unitary. Therefore, by Remark \ref{Remark 4.6} and Lemma \ref{Lemma 4.4}, $T=S_\Theta$ satisfies the hypotheses $(i)$--$(iii)$ of Theorem \ref{Theorem M}. Consequently, $S_\Theta$ satisfies all the hypotheses of Theorem \ref{Theorem M}.
\end{example}

\begin{example}[Vector valued model spaces with operator valued inner functions]
Let \(\mathcal E\) and \(\mathcal E_*\) be infinite-dimensional separable Hilbert spaces, and let
\[
        \Theta:\mathbb D\longrightarrow  B(\mathcal E,\mathcal E_*)
\]
be a purely contractive operator valued inner function. Thus the multiplication operator
\[
        M_\Theta:H^2(\mathcal E)\longrightarrow H^2(\mathcal E_*),
        \qquad
        M_\Theta f=\Theta f,
\]
is an isometry.
Let
\[
        K_\Theta
        :=
        H^2(\mathcal E_*)\ominus \Theta H^2(\mathcal E)
\]
be the corresponding vector valued model space, and assume that \(K_\Theta\) is
infinite dimensional. Define
\[
        T=S_\Theta:=\mathbf{P}_{K_\Theta}M_z|_{K_\Theta},
\]
where \(M_z\) denotes the unilateral shift on \(H^2(\mathcal E_*)\). Then \(S_\Theta\)
is a contraction. As explained in the previous example, \(S_\Theta\) is cnu. Assume further that there exists a non-empty open arc
\(\Gamma\subset\mathbb T\) such that \(\Theta\) admits an analytic continuation to a
neighbourhood of \(\Gamma\), and
\[
        \Theta(\zeta):\mathcal E\to\mathcal E_*
\]
is unitary for every \(\zeta\in\Gamma\). By the same argument as in the previous
example, we obtain $\Gamma\subset \rho(S_\Theta)$.
In particular, $\rho(S_\Theta)\cap\mathbb T\neq\emptyset$.
We also impose the compact-defect assumptions
\[
        I_{\mathcal E}-\Theta(0)^*\Theta(0)\in  K(\mathcal E), \qquad
    I_{\mathcal E_*}-\Theta(0)\Theta(0)^*\in  K(\mathcal E_*).
\]
We claim that, under these assumptions, \(S_\Theta\) satisfies all the hypotheses of
Theorem~\ref{Theorem M}.

It remains to verify that \(S_\Theta\) is essentially unitary. Recall that this means
\[
        I-S_\Theta^*S_\Theta\in  K(K_\Theta),
        \qquad
        I-S_\Theta S_\Theta^*\in K(K_\Theta).
\]

We first consider \(I-S_\Theta S_\Theta^*\). Since \(K_\Theta\) is invariant under
\(M_z^*\), we have
\[
        S_\Theta^*=M_z^*|_{K_\Theta}.
\]
Therefore, for \(f\in K_\Theta\),
\[
        S_\Theta S_\Theta^*f
        =
        \mathbf{P}_{K_\Theta}M_zM_z^*f.
\]
On \(H^2(\mathcal E_*)\), we have
\[
        M_zM_z^*=I-\mathbf{P}_{\mathcal E_*},
\]
where \(\mathbf{P}_{\mathcal E_*}\) denotes the orthogonal projection onto the initial space. Hence
\[
        S_\Theta S_\Theta^*f
        =
        \mathbf{P}_{K_\Theta}(f-f(0)).
\]
Since \(f\in K_\Theta\), this gives
\[
        (I-S_\Theta S_\Theta^*)f
        =
        \mathbf{P}_{K_\Theta}f(0).
\]
Define
\[
        C:\mathcal E_*\to K_\Theta,\qquad Cx=\mathbf{P}_{K_\Theta}x,
\]
where \(x\in\mathcal E_*\) is regarded as a constant function in
\(H^2(\mathcal E_*)\). Then
\[
        C^*f=f(0),\qquad f\in K_\Theta.
\]
Consequently,
\[
        I-S_\Theta S_\Theta^*=CC^*.
\]

We now compute \(C^*C\). Since \(\Theta\) is inner, \(M_\Theta\) is an isometry.
Therefore
\[
        \mathbf{P}_{K_\Theta}=I-M_\Theta M_\Theta^*.
\]
For \(x\in\mathcal E_*\), regarded as a constant function, we have
\[
        M_\Theta^*x=\Theta(0)^*x.
\]
Hence
\[
        \mathbf{P}_{K_\Theta}x
        =
        x-\Theta(z)\Theta(0)^*x.
\]
Taking the value at \(0\), we obtain
\[
        C^*Cx
        =
        (\mathbf{P}_{K_\Theta}x)(0)
        =
        x-\Theta(0)\Theta(0)^*x.
\]
Thus
\[
        C^*C=I_{\mathcal E_*}-\Theta(0)\Theta(0)^*.
\]
By assumption,
\[
        I_{\mathcal E_*}-\Theta(0)\Theta(0)^*\in  K(\mathcal E_*).
\]
Hence \(C^*C\) is compact. Therefore \(C\) is compact, and so
\[
        I-S_\Theta S_\Theta^*=CC^*\in K(K_\Theta).
\]

Next we consider \(I-S_\Theta^*S_\Theta\). 
 Define
\[
        B:\mathcal E\longrightarrow K_\Theta
\]
by
\[
        Bu
        =
        M_z^*M_\Theta u.
\]
Equivalently,
\[
        Bu(z)
        =
        \frac{\Theta(z)-\Theta(0)}{z}u,
        \qquad u\in\mathcal E .
\]

We first verify that \(Bu\in K_\Theta\). Since \(M_\Theta u
=\Theta(z)u\in H^2(\mathcal E_*)\), we have \(Bu\in H^2(\mathcal E_*)\). Let
\(h\in H^2(\mathcal E)\). Then
\[
\begin{aligned}
        \langle Bu,\Theta h\rangle_{H^2(\mathcal E_*)}
        &=
        \langle M_z^*M_\Theta u,M_\Theta h\rangle_{H^2(\mathcal E_*)}  \\
        &=
        \langle M_\Theta u,M_zM_\Theta h\rangle_{H^2(\mathcal E_*)}  \\
        &=
        \langle M_\Theta u,M_\Theta M_z h\rangle_{H^2(\mathcal E_*)}  \\
        &=
        \langle u,M_z h\rangle_{H^2(\mathcal E)}.
\end{aligned}
\]
Since \(u\) is a constant function and \(M_z h=zh\) has zero
constant term, we get
\[
        \langle u,M_z h\rangle_{H^2(\mathcal E)}=0.
\]
Thus \(Bu\perp \Theta H^2(\mathcal E)\), and hence \(Bu\in K_\Theta\).

Now, for \(u,v\in\mathcal E\), using the identity
\[
        \langle M_z^*F,M_z^*G\rangle_{H^2(\mathcal E)}
        =
        \langle F,G\rangle_{H^2(\mathcal E_*)}-\langle F(0),G(0)\rangle_{\mathcal E_*},
        \qquad F,G\in H^2(\mathcal E_*),
\]
we obtain
\[
\begin{aligned}
        \langle Bu,Bv\rangle_{K_\Theta}
        &=
        \langle M_z^*M_\Theta u,
        M_z^*M_\Theta v\rangle_{H^2(\mathcal E_*)}      \\
        &=
        \langle M_\Theta u,
        M_\Theta v\rangle_{H^2(\mathcal E_*)}
        -
        \langle \Theta(0)u,\Theta(0)v\rangle_{\mathcal E_*}.
\end{aligned}
\]
Since \(\Theta\) is inner, \(M_\Theta\) is an isometry. Therefore
\[
        \langle M_\Theta u,
        M_\Theta v\rangle_{H^2(\mathcal E_*)}
        =
        \langle u,v\rangle_{\mathcal E}.
\]
Hence
\[
        \langle Bu,Bv\rangle_{K_\Theta}
        =
        \langle u,v\rangle_{\mathcal E}
        -
        \langle \Theta(0)u,\Theta(0)v\rangle_{\mathcal E_*}.
\]
Thus
\[
        B^*B
        =
        I_{\mathcal E}-\Theta(0)^*\Theta(0).
\]
By assumption,
\[
        I_{\mathcal E}-\Theta(0)^*\Theta(0)\in  K(\mathcal E).
\]
Hence \(B^*B\) is compact. Since compactness of \(B^*B\) implies compactness of
\(B\), it follows that \(B\) is compact.

It remains to identify \(I-S_\Theta^*S_\Theta\) with \(BB^*\). Since
\[
        P_{K_\Theta}=I-M_\Theta M_\Theta^*,
\]
we have, for \(f\in K_\Theta\),
\[
\begin{aligned}
        S_\Theta^*S_\Theta f
        &=
        M_z^*P_{K_\Theta}M_z f  \\
        &=
        M_z^*(I-M_\Theta M_\Theta^*)M_z f  \\
        &=
        f-M_z^*M_\Theta M_\Theta^*M_z f.
\end{aligned}
\]
Therefore
\[
        (I-S_\Theta^*S_\Theta)f
        =
        M_z^*M_\Theta M_\Theta^*M_z f.
\]

Now put
\[
        g_f:=M_\Theta^*M_z f\in H^2(\mathcal E).
\]
We claim that \(g_f\) is a constant function. Indeed, using the commutation relation
\[
        M_zM_\Theta=M_\Theta M_z,
\]
and taking adjoints, we get
\[
        M_z^*M_\Theta^*=M_\Theta^*M_z^*.
\]
Hence
\[
        M_z^*g_f
        =
        M_z^*M_\Theta^*M_z f
        =
        M_\Theta^*M_z^*M_z f
        =
        M_\Theta^*f
        =
        0,
\]
because \(f\in K_\Theta\). Therefore \(g_f\in\ker M_z^*\), and so \(g_f\) is a
constant function.

Moreover, this constant is \(B^*f\). Indeed, for \(u\in\mathcal E\),
\[
\begin{aligned}
        \langle Bu,f\rangle_{K_\Theta}
        &=
        \langle M_z^*M_\Theta u,f\rangle_{H^2(\mathcal E_*)}  \\
        &=
        \langle M_\Theta u,M_z f\rangle_{H^2(\mathcal E_*)}  \\
        &=
        \langle u,M_\Theta^*M_z f\rangle_{H^2(\mathcal E)}  \\
        &=
        \langle u,g_f\rangle_{\mathcal E}.
\end{aligned}
\]
Thus \(g_f=B^*f\). Consequently,
\[
\begin{aligned}
        (I-S_\Theta^*S_\Theta)f
        &=
        M_z^*M_\Theta M_\Theta^*M_z f  \\
        &=
        M_z^*M_\Theta B^*f  \\
        &=
        BB^*f.
\end{aligned}
\]
Therefore
\[
        I-S_\Theta^*S_\Theta=BB^*.
\]
Since \(B\) is compact, it follows that
\[
        I-S_\Theta^*S_\Theta\in  K(K_\Theta).
\]

Combining the two parts, we obtain
\[
        I-S_\Theta^*S_\Theta\in  K(K_\Theta),
        \qquad
        I-S_\Theta S_\Theta^*\in  K(K_\Theta).
\]
Hence \(S_\Theta\) is essentially unitary.
Therefore \(S_\Theta\) is a cnu contraction, has a point of its resolvent set on the unit circle, and is essentially unitary. Hence, by Lemma~\ref{Lemma 4.4} and Remark~\ref{Remark 4.6}, \(S_\Theta\) satisfies all the
hypotheses of Theorem~\ref{Theorem M}.
\end{example}

\begin{example}[Cyclic roots of scalar valued model operators]
Let \(\theta\) be an inner function such that the corresponding model space
\[
        K_\theta:=H^2\ominus \theta H^2
\]
is infinite dimensional. 
Suppose further that there exists a non-empty open arc
\(\Gamma\subset\mathbb T\) such that \(\theta\) admits an analytic continuation
through a neighbourhood of \(\Gamma\). Let
\[
        S_\theta=\mathbf{P}_{K_\theta}S|_{K_\theta}
\]
be the compressed shift on \(K_\theta\). As done in Example \ref{Example 4.9}, \(S_\theta\) is a
cnu contraction, is essentially unitary, and
\[
        \sigma(S_\theta)=\sigma(\theta).
\]
Moreover, since \(\theta\) admits analytic continuation through the arc \(\Gamma\),
the points of \(\Gamma\) do not belong to \(\sigma(\theta)\). Hence
\[
        \Gamma\subset \rho(S_\theta)\cap \mathbb T .
\]

For \(m\geq 2\), set
\[
        \mathcal H_m
        =
        \underbrace{K_\theta\oplus K_\theta\oplus\cdots\oplus K_\theta}_{m\ \text{copies}}
\]
and define \(R_m\in B(\mathcal H_m)\) by
\[
        R_m(f_1,f_2,\ldots,f_m)
        =
        (f_2,f_3,\ldots,f_m,S_\theta f_1).
\]
Equivalently,
\[
        R_m=
        \begin{pmatrix}
        0&I&0&\cdots&0\\
        0&0&I&\cdots&0\\
        \vdots&\vdots&\vdots&\ddots&\vdots\\
        0&0&0&\cdots&I\\
        S_\theta&0&0&\cdots&0
        \end{pmatrix}.
\]
Then
\[
        R_m^m
        =
        S_\theta\oplus S_\theta\oplus\cdots\oplus S_\theta .
\]

We claim that \(R_m\) satisfies all the hypotheses of Theorem~\ref{Theorem M}.
First, \(R_m\) is a contraction. Indeed, for
\(f=(f_1,f_2,\ldots,f_m)\in \mathcal H_m\), we have
\[
        \|R_m f\|^2
        =
        \sum_{j=2}^{m}\|f_j\|^2+\|S_\theta f_1\|^2
        \leq
        \sum_{j=1}^{m}\|f_j\|^2
        =
        \|f\|^2.
\]

Next, \(R_m\) is cnu. Suppose, on the contrary, that \(R_m\) has a non-zero reducing
subspace \(\mathcal M\subseteq \mathcal H_m\) such that \(R_m|_{\mathcal M}\) is
unitary. Then \(R_m^m|_{\mathcal M}\) is also unitary. However,
\[
        R_m^m
        =
        S_\theta\oplus S_\theta\oplus\cdots\oplus S_\theta,
\]
and the operator on the right-hand side is cnu, because \(S_\theta\) is cnu, as shown
in Example \ref{Example 4.9}. This gives a contradiction. Hence \(R_m\) is cnu.

We next show that \(R_m\) has a resolvent point on the unit circle. By the spectral mapping theorem,
\[
        \sigma(R_m)^m
        =
        \sigma(R_m^m)
        =
        \sigma(S_{\theta}).
\]
Equivalently,
\[
        \sigma(R_m)
        =
        \{\lambda\in\mathbb C:\lambda^m\in \sigma(S_{\theta})\}.
\]
Using \(\sigma(S_\theta)=\sigma(\theta)\), this becomes
\[
        \sigma(R_m)
        =
        \{\lambda\in\mathbb C:\lambda^m\in\sigma(\theta)\}.
\]
Choose \(b\in\Gamma\). Since \(b\notin\sigma(\theta)\), choose \(a\in\mathbb T\)
such that
\[
        a^m=b.
\]
Then
\[
        a^m=b\notin\sigma(\theta)=\sigma(S_\theta).
\]
Therefore
\[
        a\notin\sigma(R_m).
\]
Hence
\[
        a\in\rho(R_m)\cap\mathbb T .
\]

Finally, we verify the essential unitarity of \(R_m\). A direct computation gives
\[
        R_m^*R_m
        =
        \operatorname{diag}(S_\theta^*S_\theta,I,\ldots,I)
\]
and
\[
        R_mR_m^*
        =
        \operatorname{diag}(I,\ldots,I,S_\theta S_\theta^*).
\]
Therefore
\[
        I-R_m^*R_m
        =
        \operatorname{diag}(I-S_\theta^*S_\theta,0,\ldots,0),
\]
and
\[
        I-R_mR_m^*
        =
        \operatorname{diag}(0,\ldots,0,I-S_\theta S_\theta^*).
\]
As done in Example \ref{Example 4.9}, both
\[
        I-S_\theta^*S_\theta
        \quad\text{and}\quad
        I-S_\theta S_\theta^*
\]
are compact. Hence
\[
        I-R_m^*R_m\in  K(\mathcal H_m),
        \qquad
        I-R_mR_m^*\in  K(\mathcal H_m).
\]
Thus \(R_m\) is essentially unitary. Consequently, by Lemma~\ref{Lemma 4.4} and
Remark~\ref{Remark 4.6}, \(R_m\) satisfies all the hypotheses of
Theorem~\ref{Theorem M}.
\end{example}

Now, in order to prove Theorem \ref{Theorem M}, we will use the following results.

The following theorem generalizes a characterization given by de Branges for spaces of scalar valued functions (see \cite{brange}) to vector valued functions based on matrix valued RK and operator valued RK that are presented in \cite{sarkar, DymJFA} and \cite{mahapatra2} respectively. 

\begin{thm} \label{Theorem C}
Let $\mathbf{H}$ be a non zero RKHS of $Y$-valued analytic functions defined on a domain $\Omega \subseteq \mathbb{C}$ which is symmetric with respect to the unit circle and contains the open unit disc. Let $K_w(z)$ denote the RK of $\mathbf{H}$ defined on $\Omega \times \Omega$. Suppose there exists a non zero $\beta \in \mathbb{D}$ such that $$K_{\beta}(z), ~ K_{\frac{1}{\bar{\beta}}}(z) ~\mbox{are Fredholm operators for all}~ z \in \Omega,$$
 and $$K_{\beta}(\beta), ~ K_{\frac{1}{\bar{\beta}}}(\tfrac{1}{\bar{\beta}}) ~\mbox{are invertible }.$$
 Let $\mathbf{H}_{\alpha} := \{ f \in \mathbf{H}: f(\alpha)=0\}$ for each point $\alpha \in \Omega$. Then the RKHS is same as a de Branges space $\mathcal{B}(\mathfrak{E})$, based on a de Branges operator $\mathfrak{E}(z)= (E_-(z), E_+(z))$ with $$K_{w}(z)= \frac{E_+(z)E_+(w)^*-E_-(z)E_-(w)^*}{\rho_w(z)}  \quad{for}~ z, w \in \Omega, z\bar{w} \neq 1,$$
 if and only if 
\begin{itemize}
\item[(1)] $R_{\beta}\mathbf{H}_{\beta} \subseteq \mathbf{H}$, $R_{\frac{1}{\bar{\beta}}}\mathbf{H}_{\frac{1}{\bar{\beta}}} \subseteq \mathbf{H}$
\item[(2)] The linear transformation $$S_{\beta}= -\bar{\beta}I+(1- \vert \beta \vert ^2)R_{\beta}: \mathbf{H}_{\beta} \rightarrow \mathbf{H}_{\frac{1}{\bar{\beta}}}$$ is an isometric isomorphism.
\end{itemize}
Moreover, in this case, the operator valued functions $E_+(z)$ and $E_-(z)$ may be specified by the formulas:
$$E_{+}(z)
=
\rho_{\beta}(z)\,K_{\beta}(z)
\big(\rho_{\beta}(\beta)\,K_{\beta}(\beta)\big)^{-1/2}$$ and $$E_{-}(z)
=
-\rho_{1/\overline{\beta}}(z)\,
K_{1/\overline{\beta}}(z)\,
\left(
-\rho_{1/\overline{\beta}}(1/\overline{\beta})\,
K_{1/\overline{\beta}}(1/\overline{\beta})
\right)^{-1/2}.$$ 
\end{thm}
\begin{proof}
The result follows from the fact that the proof of Theorem $3.1$ in \cite{mahapatra2}, originally established for spaces of entire vector valued functions, can be adapted easily to the spaces of vector valued functions that are holomorphic on the domain $\Omega$. Since, the proof technique is similar, we omit the proof details. For the matrix valued reproducing kernel setting, the analogous result may be found in \cite[Theorem 5.2]{DymJFA}.
\end{proof}

The following theorem characterizes when the multiplication operator $\mathfrak{T}$ is an isometry on its domain.  
\begin{thm}
Let $\mathbf{H}$ be a non zero RKHS of $Y$-valued analytic functions defined on a domain $\Omega \subseteq \mathbb{C}$ which is symmetric with respect to the unit circle and contains the open unit disc. Suppose that multiplication operator $\mathfrak{T}$ has domain $\mathcal{D}(\mathfrak{T})$ in $\mathbf{H}$ and $R_{\beta}\mathbf{H}_{\beta} \subseteq \mathbf{H}$ for some $\beta \in \Omega$, $\beta \notin \mathbb{T}$. Then $\mathfrak{T}$ is an isometric operator on $\mathcal{D}(\mathfrak{T})$ if and only if $$ \Vert (I-\bar{\beta}\mathfrak{T})R_{\beta}f \Vert_{\mathbf{H}}= \Vert f\Vert_{\mathbf{H}}$$  
for all $f \in \mathbf{H}_{\beta}$.
\end{thm}
\begin{proof}
Let $f,g \in \mathbf{H}_{\beta}$. Then we have
\begin{eqnarray}
&&\langle (I- \overline{\beta}\mathfrak{T})R_{\beta}f,
(I- \overline{\beta}\mathfrak{T})R_{\beta}g \rangle 
= \langle f,g \rangle \nonumber
\\[2mm] \nonumber
&\Leftrightarrow &
|\beta|^2 
\Big\langle 
\Big(
\frac{\mathfrak{T}-\beta I}{\beta}
- \mathfrak{T}\Big(1+\frac{1}{\beta}\Big)
+ I\Big(1+\frac{1}{\overline{\beta}}\Big)
\Big)R_{\beta}f,
\\
&&\qquad
\Big(
\frac{\mathfrak{T}-\beta I}{\beta}
- \mathfrak{T}\Big(1+\frac{1}{\beta}\Big)
+ I\Big(1+\frac{1}{\overline{\beta}}\Big)
\Big)R_{\beta}g
\Big\rangle
= \langle f,g \rangle \nonumber
\\[2mm]
&\Leftrightarrow &
\langle (\mathfrak{T}- \beta I)R_{\beta}f,
(\mathfrak{T}-\beta I)R_{\beta}g  \rangle  \nonumber
\\
&&+ |\beta|^2
\Big\langle 
\Big(\frac{\mathfrak{T}-\beta I}{\beta}\Big)R_{\beta}f,
\Big(
- \mathfrak{T}\Big(1+\frac{1}{\beta}\Big)
+ I\Big(1+\frac{1}{\overline{\beta}}\Big)
\Big)R_{\beta}g 
\Big\rangle \nonumber
\\
&&+ |\beta|^2
\Big\langle 
\Big(
- \mathfrak{T}\Big(1+\frac{1}{\beta}\Big)
+ I\Big(1+\frac{1}{\overline{\beta}}\Big)
\Big)R_{\beta}f, \nonumber
\\
&&\qquad
\Big(
- \mathfrak{T}\Big(1+\frac{1}{\beta}\Big)
+ I\Big(1+\frac{1}{\overline{\beta}}\Big)
\Big)R_{\beta}g
\Big\rangle  \label{M}
\\
&&+ |\beta|^2 
\Big\langle 
\Big(
- \mathfrak{T}\Big(1+\frac{1}{\beta}\Big)
+ I\Big(1+\frac{1}{\overline{\beta}}\Big)
\Big)R_{\beta}f,
\Big(\frac{\mathfrak{T}-\beta I}{\beta}\Big)R_{\beta}g
\Big\rangle
= \langle f,g \rangle  \nonumber .
\end{eqnarray}
Observe that
\begin{equation}
(\mathfrak{T}-\beta I)R_{\beta}f
= f. \label{1}
\end{equation}
Indeed, pointwise we have $$((\mathfrak{T}-\beta I)R_{\beta}f)(z)
=
z(R_{\beta}f)(z)- \beta(R_{\beta}f)(z)
=
f(z).$$
Furthermore,
\begin{equation}\label{3}
\begin{aligned}
&\Big\langle 
\Big(\frac{\mathfrak{T}-\beta I}{\beta}\Big)R_{\beta}f,
\Big(
- \mathfrak{T}\Big(1+\frac{1}{\beta}\Big)
+ I\Big(1+\frac{1}{\overline{\beta}}\Big)
\Big)R_{\beta}g 
\Big\rangle
\\
&=
-\frac{1}{\beta}\Big(1+\frac{1}{\overline{\beta}}\Big)
\langle \mathfrak{T}R_{\beta}f, \mathfrak{T}R_{\beta}g \rangle
\\
&\quad
+\Big(1+\frac{1}{\overline{\beta}}\Big)
\langle R_{\beta}f, \mathfrak{T}R_{\beta}g \rangle
\\
&\quad
+\frac{1}{\beta}\Big(1+\frac{1}{\beta}\Big)
\langle \mathfrak{T}R_{\beta}f,R_{\beta}g \rangle
\\
&\quad
-\Big(1+\frac{1}{\beta}\Big)
\langle R_{\beta}f, R_{\beta}g \rangle .
\end{aligned}
\end{equation}

Similarly,
\begin{equation}\label{4}
\begin{aligned}
&\Big\langle 
\Big(
- \mathfrak{T}\Big(1+\frac{1}{\beta}\Big)
+ I\Big(1+\frac{1}{\overline{\beta}}\Big)
\Big)R_{\beta}f,
\\
&\qquad
\Big(
- \mathfrak{T}\Big(1+\frac{1}{\beta}\Big)
+ I\Big(1+\frac{1}{\overline{\beta}}\Big)
\Big)R_{\beta}g
\Big\rangle
\\
&=
\left|1+\frac{1}{\beta}\right|^2 
\langle \mathfrak{T}R_{\beta}f, \mathfrak{T}R_{\beta}g \rangle
\\
&\quad
-(1+\tfrac{1}{\beta})^2 
\langle \mathfrak{T}R_{\beta}f, R_{\beta}g \rangle
\\
&\quad
-(1+\tfrac{1}{\overline{\beta}})^2 
\langle R_{\beta}f, \mathfrak{T}R_{\beta}g \rangle
\\
&\quad
+\left|1+\frac{1}{\beta}\right|^2 
\langle R_{\beta}f, R_{\beta}g \rangle .
\end{aligned}
\end{equation}

Also,
\begin{equation}\label{5}
\begin{aligned}
&\Big\langle 
\Big(
- \mathfrak{T}\Big(1+\frac{1}{\beta}\Big)
+ I\Big(1+\frac{1}{\overline{\beta}}\Big)
\Big)R_{\beta}f,
\Big(\frac{\mathfrak{T}-\beta I}{\beta}\Big)R_{\beta}g
\Big\rangle
\\
&=
-\frac{1}{\overline{\beta}}(1+\tfrac{1}{\beta})
\langle \mathfrak{T}R_{\beta}f, \mathfrak{T}R_{\beta}g \rangle
\\
&\quad
+(1+\tfrac{1}{\beta})
\langle \mathfrak{T}R_{\beta}f, R_{\beta}g \rangle
\\
&\quad
+\frac{1}{\overline{\beta}}
(1+\tfrac{1}{\overline{\beta}})
\langle R_{\beta}f, \mathfrak{T}R_{\beta}g \rangle
\\
&\quad
-(1+\tfrac{1}{\overline{\beta}})
\langle R_{\beta}f, R_{\beta}g \rangle .
\end{aligned}
\end{equation}
Substituting \eqref{1}--\eqref{5} into \eqref{M}, we obtain
\[
\langle (I- \overline{\beta}\mathfrak{T})R_{\beta}f,
(I- \overline{\beta}\mathfrak{T})R_{\beta}g \rangle 
= \langle f,g \rangle 
\iff 
\langle \mathfrak{T}R_{\beta}f, \mathfrak{T}R_{\beta}g \rangle
=
\langle R_{\beta}f, R_{\beta}g \rangle .
\]
Finally, we show that $$\mathcal{D}(\mathfrak{T})= R_{\beta}\mathbf{H}_{\beta} \iff R_{\beta}\mathbf{H}_{\beta} \subseteq \mathbf{H}.$$
The forward implication is immediate. For the converse, let 
$f \in \mathbf{H}_{\beta}$. Then
\begin{align*}
(R_{\beta}f)(z)
&= \frac{f(z)}{z-\beta},
\\
\Rightarrow\;
z(R_{\beta}f)(z)
&=
f(z)+\beta (R_{\beta}f)(z),
\\
\Rightarrow\;
R_{\beta}f 
&\in \mathcal{D}(\mathfrak{T}).
\end{align*}
Thus,
\[
R_{\beta}\mathbf{H}_{\beta} 
\subseteq  
\mathcal{D}(\mathfrak{T}).
\]
Now, to prove $\mathcal{D}(\mathfrak{T}) \subseteq R_{\beta}\mathbf{H}_{\beta}$, let $f \in \mathcal{D}(\mathfrak{T})$. Then
\[
R_{\beta}(\mathfrak{T}-\beta I)f= f
\]
and
\[
(\mathfrak{T}- \beta I)f(\beta)=0.
\]
This completes the proof.
\end{proof}

The following is an easy corollary of the above theorem that will be used in the proof of Theorem \ref{Theorem M}.

\begin{cor}\label{cor} Let $\mathbf{H}$ be a non zero RKHS of $Y$-valued analytic functions defined on a domain $\Omega \subseteq \mathbb{C}$ which is symmetric with respect to the unit circle and contains the open unit disc. Suppose that multiplication operator $\mathfrak{T}$ has domain $\mathcal{D}(\mathfrak{T})$ in $\mathbf{H}$ and $R_{\beta}\mathbf{H}_{\beta} \subseteq \mathbf{H}$, $R_{\frac{1}{\bar{\beta}}}\mathbf{H}_{\frac{1}{\bar{\beta}}} \subseteq \mathbf{H}$ for some $\beta \in \mathbb{D}$. Then $\mathfrak{T}$ is an isometric operator on $\mathcal{D}(\mathfrak{T})$ if and only if the operator $(I-\bar{\beta}\mathfrak{T})R_{\beta}$ maps $\mathbf{H}_{\beta}$ isometrically onto $\mathbf{H}_{\frac{1}{\bar{\beta}}}$.
\end{cor}

For the contraction operators $V_0$ and $V$, defined in \eqref{V0} and \eqref{V}, respectively, consider the transformation
\begin{equation}
U_{zw}= I+(z-w)(\widetilde{V}-zI)^{-1},
\end{equation}
where $\widetilde{V}$ is a unitary extension of the operator $V$. This transformation is analogous to the generalized Cayley transform for symmetric operators; see, for instance, \cite[Chapter 1, Section 2]{kreinlect}.

The following lemma records some basic properties of the transformation $U_{zw}$ that will be used in the sequel.
\begin{lemma}\label{Lemma Cayley} For all $z,w \in \Omega_a \setminus \mathbb{T}$ with 
$z \neq 0$ and $w \neq 0$, the following statements hold:
\begin{itemize}
\item[(1)]$U_{zw}^{-1}= U_{wz}$.
\item[(2)] $U_{zw}$ is a one-to-one and onto map from $M_z$ onto $M_w$.
\item[(3)] $U_{zw}$ is a one-to-one and onto map from 
$M_{\frac{1}{\overline{w}}}^{\perp}$ onto $M_{\frac{1}{\overline{z}}}^{\perp}$.
\end{itemize}
\end{lemma}
\begin{proof}\textbf{(1)} \begin{eqnarray*}
U_{zw}^{-1}&=& \big(I+(z-w)(\widetilde{V}-zI)^{-1}\big)^{-1}\\
&=& \big((\tilde{V}-z+z-w)(\tilde{V}-z)^{-1} \big)^{-1}\\
&=&\big((\tilde{V}-w)(\tilde{V}-z)^{-1}\big)^{-1}\\
&=& (\tilde{V}-z)(\tilde{V}-w)^{-1}= U_{wz}
\end{eqnarray*}
\textbf{(2)}
Let $g \in (\ker V)^{\perp}$. Then $Vg=\widetilde{V}g$. Hence, we have
\begin{eqnarray*}
U_{zw}(V_0-zI)g
&=& \big(I+(z-w)(\widetilde{V}-zI)^{-1}\big)(V_0-zI)g\\
&=&(V_0-zI)g+(z-w)g\\
&=&(V_0-wI)g .
\end{eqnarray*}
Let $h \in M_w$. Then there exists $g \in (\ker V)^{\perp}$ such that $h=(V_0-wI)g $. If we set $f=(V_0-zI)g,$ then from the above equality we obtain $U_{zw}f=h$. Therefore, $U_{zw}$ is onto. Next, we show that $U_{zw}$ is one-one. Let $f \in M_z$ be such that $U_{zw}f=0 $. Since $f \in M_z$, there exists $g \in H$ such that $f=(V_0-zI)g$. Hence,
\begin{eqnarray*}
& & U_{zw}(V_0-zI)g=0 \\
&\Rightarrow &(V_0-wI)g=0 .
\end{eqnarray*}
Since $w$ is a point of regular type of $V_0$, it follows that $g=0$. Consequently, $f=0$. Thus, $U_{zw}$ is one-one.

\textbf{(3)} Let $\phi \in M_{\frac{1}{\overline{w}}}^{\perp}$. Then for all $g \in M_{\frac{1}{\overline{z}}}$, we have \begin{eqnarray*}
\langle U_{zw} \phi,g \rangle &=& \langle \phi, U_{zw}^*g\rangle\\
&=& \langle\phi, (\tilde{V}^{-1}-\overline{w})(\tilde{V}^{-1}-\overline{z})^{-1}g \rangle\\
&=& \langle\phi, (\tilde{V}^{-1}-\overline{w})(\tilde{V}^{-1}-\overline{z})^{-1}(V_0- \tfrac{1}{\overline{z}}) f \rangle\\
&=& \langle\phi, (\tilde{V}^{-1}-\overline{w})\tilde{V}(I- \overline{z}\tilde{V})^{-1}(V_0- \tfrac{1}{\overline{z}}) f \rangle\\ 
&=&\frac{-\overline{w}}{z}\langle \phi, (I-\overline{w}\tilde{V})f \rangle=0
\end{eqnarray*}
where $f$ is such that $g= (V_0- \tfrac{1}{\overline{z}})f$.
This implies that $$U_{zw}M_{\frac{1}{\overline{w}}}^{\perp} \subset M_{\frac{1}{\overline{z}}}^{\perp}. $$ Similarly, we obtain $$U_{wz}M_{\frac{1}{\overline{z}}}^{\perp} \subset M_{\frac{1}{\overline{w}}}^{\perp}. $$
By \textbf{(1)}, we get that $U_{zw}$ is a one-to-one and onto map from 
$M_{\frac{1}{\overline{w}}}^{\perp}$ onto $M_{\frac{1}{\overline{z}}}^{\perp}$.
\end{proof}

Now we prove Theorem~\ref{Theorem M}.
\begin{proof}[Proof of Theorem \ref{Theorem M} ] In order to prove the theorem, we will use the characterization of vector valued de Branges spaces $\mathcal{B}(\mathfrak{E})$ given in Theorem \ref{Theorem C}. The proof is divided into the following steps. In step $1$, we show that $K_{\beta}(z)$ and $ K_{\frac{1}{\bar{\beta}}}(z)$ are Fredholm operators for all $z \in \Omega_a$. In step $2$, we show that the operators $K_{\beta}(\beta)$ and $ K_{\frac{1}{\bar{\beta}}}(\frac{1}{\bar{\beta}})$ are invertible. In steps $3$ and $4$, we verify conditions $(1)$ and $(2)$ of Theorem \ref{Theorem C}, respectively.\\
\textbf{Step 1.} Since $\mathbf{H}$ is a RKHS, the RK is given by $K_{w}(z)= \delta_z\delta_w^*$, for all $z,w \in \Omega_a$. Here $\delta_z: \mathbf{H} \rightarrow Y$ denotes the pointwise evaluation operator defined by $$\delta_z(f_Y)= f_Y(z)= P_Y(z)f.$$ 
Hence, \begin{equation} \label{Equation 4.1}\rng(\delta_z)=Y~ \mbox{and}~ \ker(\delta_z)=\{ f_Y : f \in (V_0-zI)(\ker V)^{\perp} \}.
\end{equation} 
This implies \begin{equation} \label{Equation 4.2}\ker(\delta_z^*)=\{ 0 \}~ \mbox{and}~ \rng(\delta_z^*)=\{ f_Y : f \in H \ominus (V_0-zI)(\ker V)^{\perp}  \}.
\end{equation}
Thus, \begin{eqnarray*}\dim(\ker K_{\beta}(z))
&=&\dim(\ker \delta_z \delta_{\beta}^*)\\
&=& \dim(\ker \delta_{\beta}^*)+\dim(\ker \delta_z \cap \rng \delta_{\beta}^*)\\
&=& \dim(\ker \delta_z \cap \mathrm{rng}~\delta_{\beta}^*)\\
&=& \dim(M_z \cap M_{\beta}^{\perp}).\end{eqnarray*}

Similarly, $$\dim(\ker K_{\beta}(z)^*)=\dim(\ker K_{z}(\beta))=\dim(M_{\beta} \cap M_{z}^{\perp}).$$
By, Lemma \ref{Lemma Cayley}, the map $$U_{z \beta}: M_z \cap M_{\frac{1}{\overline{\beta}}}^{\perp} \rightarrow M_{\beta} \cap M_{\frac{1}{\overline{z}}}^{\perp}$$ is bijective for all $z \in \Omega_a \setminus \mathbb{T}$, $z \neq 0$. Hence, for all such $z$, we have the following equality 
\begin{equation}\label{4.5}
\dim (M_z \cap M_{\frac{1}{\overline{\beta}}}^{\perp} )= \dim(M_{\beta} \cap M_{\frac{1}{\overline{z}}}^{\perp})
\end{equation}
Similarly, the map
 $$U_{z \frac{1}{\overline{\beta}}}: M_z \cap M_{\beta}^{\perp} \rightarrow M_{\frac{1}{\overline{\beta}}} \cap M_{\frac{1}{\overline{z}}}^{\perp}$$
  is bijective for all $z \in \Omega_a \setminus \mathbb{T}$, $z \neq 0$. Hence, for all such $z$, we have
  \begin{equation}\label{eq4.6}
\dim (M_z \cap M_{\beta}^{\perp})= \dim(M_{\frac{1}{\overline{\beta}}} \cap M_{\frac{1}{\overline{z}}}^{\perp})\end{equation}
If $z \in \Omega_a \cap \mathbb{T}$, then by Theorem \ref{Theorem D}, we have the decomposition \begin{equation} \label{0}\mathcal{H}= M_{\beta} \dotplus M_{z}^{\perp}= M_{\frac{1}{\overline{\beta}}} \dotplus M_{z}^{\perp}.
\end{equation}
This implies \begin{equation}\label{4.8}
M_{\beta} \cap M_{z}^{\perp}= \{0 \} ~\mbox{and}~ M_{\frac{1}{\overline{\beta}}} \cap M_{z}^{\perp} =\{0 \}~\mbox{for all}~ z \in \Omega_a \cap \mathbb{T}.
\end{equation} Moreover, it is known that for closed subspaces $M$,$N$ of  a Banach space, $$(M+N)^{\perp}= M^{\perp} \cap N^{\perp}.$$ 
Using this fact and \eqref{0}, we obtain $\{0\} = \mathcal{H}^{\perp}= \big( M_{\beta} + M_{z}^{\perp}\big)^{\perp}= M_{\beta}^{\perp} \cap (M_z^{\perp})^{\perp}$ and  $\{0\} = \mathcal{H}^{\perp}= \big( M_{\frac{1}{\overline{\beta}}} + M_{z}^{\perp}\big)^{\perp}= M_{\frac{1}{\overline{\beta}}}^{\perp} \cap (M_z^{\perp})^{\perp}$. Hence \begin{equation}\label{4.9}
M_{\beta}^{\perp} \cap M_{z}= \{0 \} ~\mbox{and}~ M_{\frac{1}{\overline{\beta}}}^{\perp} \cap M_{z} =\{0 \} ~\mbox{for all}~ z \in \Omega_a \cap \mathbb{T}.
\end{equation} 
Now, by conditions i), ii), iii) of the hypothesis together with \eqref{eq4.6}, \eqref{4.8} and \eqref{4.9}, we conclude that $\dim(\ker K_{\beta}(z))$ and $\dim(\ker K_{\beta}(z)^*)$ are finite for all $z \in \Omega_a$. To show that $K_{\beta}(z)$ is Fredholm for all $z\in \Omega_a$, it remains to prove that $\rng{K_{\beta}(z)}$ is closed for all $z \in \Omega_a$, or equivalently, $\rng{K_{z}(\beta)}$ is closed for all $z \in \Omega_a$. By \cite[Corollary 2.5]{prod}, we have that \begin{eqnarray*} \rng K_{z}(\beta) \text{ is closed}\iff  \ker\delta_{\beta}+ \rng \delta_{z}^*\text{ is closed} \iff M_{\beta} +M_z^{\perp}\text{ is closed}.\end{eqnarray*} By the closedness condition in i) of hypothesis together with \eqref{0}, we obtain that $\rng K_{\beta}(z)$ is closed for all $z \in \Omega_a$. Similarly, using conditions i), ii), iii) of the hypothesis and by observation in \eqref{4.5}, \eqref{eq4.6}, \eqref{0}, \eqref{4.8} and \eqref{4.9}, it can be shown that $K_{\frac{1}{\overline{\beta}}}(z)$ is Fredholm for all $z \in \Omega_a$.\\
$\textbf{Step 2.}$ By Step 1, the operator $ K_{\beta}(\beta)$ is Fredholm. Moreover, its index is zero since $$ \dim(\ker K_{\beta}(\beta)^*)=\dim(\ker K_{\beta}(\beta))= \dim(M_{\beta} \cap M_{\beta}^{\perp})= 0 .$$
Now, by part $(2)$ of Theorem \ref{Theorem 2.5}, we get that the operator $ K_{\beta}(\beta)$ is invertible. Similarly, it can be shown that $ K_{\frac{1}{\overline{\beta}}}(\frac{1}{\overline{\beta}})$ is also invertible.\\
$\textbf{Step 3.}$ The space $\mathbf{H}$ is $R_{z}$ invariant for all $z \in \Omega_a$. To see this, let $g$ and $g'$ be such that $$f= (V_0-wI)g+P_Y(w)f= (V_0-zI)g'+P_Y(z)f,$$
which implies that $$g'=(V_0-wI)\frac{g-g'}{w-z}+\frac{P_Y(w)f-P_Y(z)f}{w-z}.$$ Thus, $$(R_zf_Y)(w)= \frac{f_Y(w)-f_Y(z)}{w-z}= \frac{P_Y(w)f-P_Y(z)f}{w-z}= P_Y(w)g'= g'_Y(w).$$
Hence, in particular, for the points $\beta$ and $\overline{\beta}$, and the spaces $\mathbf{H}_{\beta}$ and $\mathbf{H}_{\frac{1}{\overline{\beta}}}$, condition $(1)$ holds.\\
$\textbf{Step 4.}$ The operator $V_0$ is isometry on $H \oplus \{0 \}$ and is unitarily equivalent to the multiplication operator $\mathfrak{T}$ on $ \mathcal{D}(\mathfrak{T})= \Psi(H \oplus \{0 \})$, where $\Psi$ is as given by $\eqref{Psi}$. Hence, the multiplication operator $\mathfrak{T}$ is isometric operator on $\Psi(H \oplus \{0 \})$. Now, using Corollary \ref{cor}, we have that the operator $(I-\bar{\beta}\mathfrak{T})R_{\beta}$ maps $\mathbf{H}_{\beta}$ isometrically onto $\mathbf{H}_{\frac{1}{\bar{\beta}}}$. Since $$S_{\beta}f= \big( -\bar{\beta}I+(1- \vert \beta \vert ^2)R_{\beta}\big)f= (I-\bar{\beta}\mathfrak{T})R_{\beta}f,$$ for all $f \in \mathbf{H}_{\beta}$, we get that the condition $(2)$ holds.
\end{proof}

\section{Connections between the Sz.-Nagy-Foias characteristic function of $T$, $P_Y(z)$ and the corresponding de Branges reproducing kernel }
\label{Section 5}
In this section, we show that the two operator valued analytic functions $\Theta_T(z)$ and $P_Y(z)$ coincide on $\mathbb{D}$.   
First, we recall the definition of coincidence on an open set $\Omega$ in $\mathbb{C}$. For each $\lambda \in \Omega$, let $R(\lambda) \in B(X_1,Y_1)$ and $Q(\lambda) \in B(X_2, Y_2)$, be bounded linear operators, where $X_1$, $X_2$, $Y_1$, and $Y_2$ are Hilbert spaces. The operator valued functions $R(.)$ and $Q(.)$ are said to coincide on $\Omega$ if there exist unitary operators $A \in  B(X_1,X_2)$ and $B \in B(Y_1,Y_2)$ such that: $$R(\lambda)= A Q(\lambda)B^{-1}, \quad \lambda\in \Omega.$$

For the cnu contraction operator $T$ on a Hilbert space $H$ such that $a \in \mathbb{T} \cap \rho(T)$, define the operator $S_a$ on $H$ by
\[
S_a=-(T^*-\overline{a})^{-1}D_T\Big(I+D_T(T-a)^{-1}(T^*-\overline{a})^{-1}D_T\Big)^{-1/2}.
\]
Let us evaluate $I-S_a^*S_a$:
\begin{align*}
I-S_a^*S_a
&= I-\Big((T^*-\overline{a})^{-1}D_T
\big(I+D_T(T-a)^{-1}(T^*-\overline{a})^{-1}D_T\big)^{-1/2}\Big)^* \\
&\qquad \times (T^*-\overline{a})^{-1}D_T
\big(I+D_T(T-a)^{-1}(T^*-\overline{a})^{-1}D_T\big)^{-1/2}\\[2mm]
&= I-
\big(I+D_T(T-a)^{-1}(T^*-\overline{a})^{-1}D_T\big)^{-1/2} \\
&\qquad \times D_T(T-a)^{-1}(T^*-\overline{a})^{-1}D_T \\
&\qquad \times 
\big(I+D_T(T-a)^{-1}(T^*-\overline{a})^{-1}D_T\big)^{-1/2}\\[2mm]
&=
\big(I+D_T(T-a)^{-1}(T^*-\overline{a})^{-1}D_T\big)^{-1/2} \\
&\qquad \times 
\big(I+D_T(T-a)^{-1}(T^*-\overline{a})^{-1}D_T  \\
&\qquad\qquad - D_T(T-a)^{-1}(T^*-\overline{a})^{-1}D_T\big) \\
&\qquad \times
\big(I+D_T(T-a)^{-1}(T^*-\overline{a})^{-1}D_T\big)^{-1/2}\\[2mm]
&=
\big(I+D_T(T-a)^{-1}(T^*-\overline{a})^{-1}D_T\big)^{-1}.
\end{align*}
Since 
\[
\big(I+D_T(T-a)^{-1}(T^*-\overline{a})^{-1}D_T\big)^{-1}\geq 0,
\]
it follows that $I-S_a^*S_a \geq 0$. Therefore,
\[
S_a^*S_a \leq I,
\]
which shows that $S_a$ is a contraction.
Consider the unitary operator matrix on the Hilbert space $\mathcal{H}$ given by the Julia operator matrix corresponding to the contraction $S_a$:
\begin{equation}
\Gamma_a=
\begin{bmatrix}
S_a & D_{S_a^*}\\
D_{S_a} & -S_a^*
\end{bmatrix}.
\end{equation}
We identify $H$ with the subspace $H\oplus\{0\}\subset \mathcal{H}$, and $\mathcal{D}_T$, $\mathcal{D}_{T^*} \subseteq H$ with $\mathcal{D}_T \oplus \{0 \}$, $\mathcal{D}_{T^*} \oplus \{0 \} \subseteq \mathcal{H}$.

For every $f\in H$ and $z\in\mathbb{D}$, we have
\begin{eqnarray}\label{5.2}
\Gamma_a \Theta_T(z)D_Tf
&=& \Gamma_a
\begin{bmatrix}
\Theta_T(z)D_Tf\\
0
\end{bmatrix} \nonumber\\
&=& \Gamma_a
\begin{bmatrix}
D_{T^*}(I-zT^*)^{-1}(z-T)f\\
0
\end{bmatrix} \nonumber\\
&=&
\begin{bmatrix}
S_aD_{T^*}(I-zT^*)^{-1}(z-T)f\\
D_{S_a}D_{T^*}(I-zT^*)^{-1}(z-T)f
\end{bmatrix} \nonumber\\
&=&
\begin{bmatrix}
-(T^*-\overline{a})^{-1}D_T h\\
h
\end{bmatrix},
\end{eqnarray}
where
\begin{equation}\label{h}
h=
\Big(I+D_T(T-a)^{-1}(T^*-\overline{a})^{-1}D_T\Big)^{-1/2}
D_{T^*}(I-zT^*)^{-1}(z-T)f.
\end{equation}
The element \[
\begin{bmatrix}
0\\
D_T f
\end{bmatrix}\in \mathcal{H}
\] admits the following direct sum decomposition (see the proof of Theorem \ref{Theorem 5.1}):  
\begin{equation}\label{5.4} \begin{bmatrix} 
	0  \\
	D_Tf \\
	\end{bmatrix}= \begin{bmatrix} 
	(T-zI)j  \\
	D_Tj \\
	\end{bmatrix}+ \begin{bmatrix} 
	-(T^*-\overline{a})^{-1}D_T h \\
	h\\
	\end{bmatrix},
	\end{equation}
where \begin{equation}\label{j}j= (I+ \vert z \vert^2-zT^*-\overline{z}T)^{-1} \Big( (T^*-\overline{z}) (T^*-\overline{a})^{-1} D_Th+D_T (D_Tf-h) \Big).\end{equation}
The set 
\[
Y=M_a^{\perp}=((V_0-aI)(\ker V)^{\perp})^{\perp}
\]
is given by
\begin{eqnarray}
M_a^{\perp}
&=& \left\{
\begin{bmatrix}
h_1\\
h_2
\end{bmatrix}\in \mathcal{H}:
\left\langle
\begin{bmatrix}
h_1\\
h_2
\end{bmatrix},
\begin{bmatrix}
(T-aI)k\\
D_Tk
\end{bmatrix}
\right\rangle =0,\ \text{for all } k\in H
\right\} \nonumber\\
&=& \left\{
\begin{bmatrix}
h_1\\
h_2
\end{bmatrix}\in \mathcal{H}:
\langle (T^*-\overline{a})h_1+D_T h_2,\, k\rangle =0,\ \text{for all } k\in H
\right\} \nonumber\\
&=& \left\{
\begin{bmatrix}
-(T^*-\overline{a})^{-1}D_T h_2\\
h_2
\end{bmatrix}
:\ h_2\in H
\right\} \label{eqn 5.6}
\end{eqnarray}
Thus, in the decomposition \eqref{5.4}, the element
\[
\begin{bmatrix}
(T-zI)j\\
D_T j
\end{bmatrix}
\in (V_0-zI)(\ker V)^{\perp}
\]
and
\[
\begin{bmatrix}
-(T^*-\overline{a})^{-1}D_T h\\
h
\end{bmatrix}
\in Y.
\]
By Theorem~\ref{Theorem D}, such a decomposition is unique. Hence,
\begin{equation}\label{5.6}
P_Y(z)
\begin{bmatrix}
0\\
D_T f
\end{bmatrix}
=
\begin{bmatrix}
-(T^*-\overline{a})^{-1}D_T h\\
h
\end{bmatrix},
\end{equation}
where \(h\) is as defined in \eqref{h}. 
Define a unitary operator $J$ on the Hilbert space $\mathcal{H}$ by
\begin{equation}\label{J}
J=
\begin{bmatrix}
0 & I\\
I & 0
\end{bmatrix}.
\end{equation}

Thus, by \eqref{5.2}, \eqref{5.6}, and \eqref{J}, we obtain that for every $f\in H$,
\begin{equation}
\Gamma_a \Theta_T(z)D_T f
=
P_Y(z)J
\begin{bmatrix}
D_T f\\
0
\end{bmatrix}.
\end{equation}

We summarize the above discussion in the following theorem of this section which provides the connection between the Sz.-Nagy-Foias characteristic function of $T$ and $P_Y(z)$ .	
\begin{thm}\label{Theorem 5.1}
Let $T$ be a cnu contraction operator on a Hilbert space $H$ such that a point $a \in \mathbb{T}$ belongs to $\rho(T)$. Then, for every $z \in \mathbb{D}$, the characteristic operator valued function $\Theta_T(z) \in B(\mathcal{D}_T, \mathcal{D}_{T^*})$ coincides with the operator valued function $P_Y(z)\big|_{\{0\} \oplus \mathcal{D}_T}$, where $P_Y(z)$ is as defined in \eqref{4.1}.
\end{thm}
\begin{proof}
It remains to show that the decomposition given in \eqref{5.4} holds. 
For this purpose, define the operator $F_z: H \rightarrow H \oplus H$, for each $z \in \mathbb{D}$ given by
\[
F_z j :=
\begin{bmatrix}
(T-zI)j\\
D_T j
\end{bmatrix}, \quad j \in H
\]
The following equalities hold for all $j \in H$:
\begin{eqnarray*}
\left\langle j, F_z^*
\begin{bmatrix}
x\\
y
\end{bmatrix}
\right\rangle
&=&
\left\langle F_z j,
\begin{bmatrix}
x\\
y
\end{bmatrix}
\right\rangle \\
&=&
\left\langle
\begin{bmatrix}
(T-zI)j\\
D_T j
\end{bmatrix},
\begin{bmatrix}
x\\
y
\end{bmatrix}
\right\rangle \\
&=&
\langle (T-zI)j, x \rangle + \langle D_T j, y \rangle \\
&=&
\langle j, (T^*-\overline{z})x + D_T y \rangle .
\end{eqnarray*}	 This implies that the adjoint operator $F_z^*$ is given by $$F_z^*\begin{bmatrix} 
	x \\
	y \\
	\end{bmatrix}= (T^*-\overline{z})x+D_Ty.$$
	Hence, $$F_z^*F_z= I+\vert z \vert^2- \overline{z}T-zT^*.$$ Consider the following:
	\begin{eqnarray*}
	\langle F_z^*F_z j,j \rangle &=& \langle (I+\vert z \vert^2- \overline{z}T-zT^*)j,j \rangle \\
	&=&(I+\vert z \vert^2)\Vert j \Vert^2- \langle (\overline{z}T+zT^*)j,j \rangle\\
	&=& (I+\vert z \vert^2)\Vert j \Vert^2-2 \Re (\overline{z} \langle Tj,j \rangle)\\
	&\geq & (I+\vert z \vert^2)\Vert j \Vert^2-2 \vert z \vert \vert\langle Tj,j \rangle \vert \\
	&\geq & (I+\vert z \vert^2)\Vert j \Vert^2-2 \vert z \vert \Vert j \Vert^2 \\
	&=& (I-\vert z \vert )^2 \Vert j \Vert^2
	\end{eqnarray*}
Thus, for each $z\in\mathbb{D}$, the operator $F_z$ is bounded below. Consequently, the operator $F_z^*F_z$ is invertible. We now determine $j$ for which the following equality holds \begin{equation}\label{5.10}
F_zj= \begin{bmatrix} 
	(T^*-\overline{a})^{-1}D_T h \\
	D_Tf -h\\
	\end{bmatrix},
\end{equation}
where $h$ is as defined in \eqref{h}. Premultiplying both sides of \eqref{5.10} by $(F_z^*F_z)^{-1}F_z^*$, we obtain $$j= (F_z^*F_z)^{-1}F_z^*\begin{bmatrix} 
	(T^*-\overline{a})^{-1}D_T h \\
	D_Tf -h\\
	\end{bmatrix}$$ 
Since the decomposition in \eqref{5.4} is unique, this is precisely the element $j$ given in \eqref{j}. The remainder of the proof now follows from the above discussions.
\end{proof}	

\begin{rmk}\label{Remark 5.2}
Observe that for $f \in \mathcal{H}$ and $z \in \Omega_a$, we have  $$P_Y(z)f= f_Y(z)= \delta_z(f_Y)= \delta_z(\Psi f).$$
This implies that $$P_Y(z)= \delta_z\Psi, \quad z \in \Omega_a.$$
Consequently, for all $z,w \in \Omega_a$, we have
\begin{eqnarray*}
K_w(z)= \delta_z \delta_w^*= P_Y(z) \Psi^{-1} \Psi P_Y(w)^*= P_Y(z) P_Y(w)^*.
\end{eqnarray*}
\end{rmk}
\begin{rmk}\label{Remark 5.3}
From Theorem \ref{Theorem 5.1} and the expression of kernel function given in Theorem \ref{Theorem C}, we obtain the following for all $z, w \in \mathbb{D}$,
 \begin{eqnarray*}\Gamma_a^*~\frac{E_+(z)E_+(w)^*-E_-(z)E_-(w)^*}{\rho_w(z)}~\Gamma_a \big|_{\mathcal{D}_{T^*} \oplus \{0 \}} &=& \Gamma_a^*P_Y(z) P_Y(w)^*\Gamma_a \big|_{\mathcal{D}_{T^*} \oplus \{0 \}} \\
&=& \Theta_T(z)\Theta_T(w)^* .
\end{eqnarray*}
\end{rmk}
This establishes the connection between the operator valued functions $\Theta_T(z)$, $P_Y(z)$ and $K_w(z)$.

In the course of constructing a functional model for the class of cnu contractions having at least one resolvent point on the unit circle, realized in a de Branges space corresponding to a de Branges operator, we are naturally led to investigate the question of unitary equivalence of two such contractions, say $T_1$ and $T_2$, and its relationship with the corresponding de Branges spaces. To this end, we observe that the de Branges quotient $E_+^{-1}E_-$ serves as a complete unitary invariant. To formulate the result precisely, for \(j=1,2\), let
\[
V_{0,j}:H\oplus \{0\}\to H\oplus H, \qquad
V_{0,j}=
\begin{bmatrix}
T_j & 0\\
D_{T_j} & 0
\end{bmatrix},
\]
and let
\[
Y_j:=\Bigl((V_{0,j}-aI)(H\oplus\{0\})\Bigr)^\perp.
\]
We shall prove the following theorem.
\begin{thm}\label{thm:unitary-equivalence-via-EplusinvEminus}
Let $T_1,T_2\in B(H)$ be two cnu contraction operators on the same Hilbert space $H$. Assume that there exist $
a\in \rho(T_1)\cap \rho(T_2)\cap \mathbb{T}$
and $\beta\in \mathbb{D}$ such that the hypotheses {\rm (i)}, {\rm (ii)}, {\rm (iii)} of Theorem~\ref{Theorem M} are satisfied for both $T_1$ and $T_2$. Let $
\mathcal{B}(\mathfrak{E})$ and $\mathcal{B}(\mathfrak{F})$
be the corresponding vector valued de Branges spaces obtained from $T_1$ and $T_2$, respectively, where
\[
\mathfrak{E}=(E_-,E_+), \qquad \mathfrak{F}=(F_-,F_+).
\]
Let $\mathfrak{T}_1$ and $\mathfrak{T}_2$ denote the corresponding multiplication operators by the independent variable on $\mathcal{B}(\mathfrak{E})$ and $\mathcal{B}(\mathfrak{F})$, respectively. Then the following are equivalent.
\begin{enumerate}
\item[(1)] $T_1$ and $T_2$ are unitarily equivalent.

\item[(2)] There exist unitary operators
\[
\Delta:Y_2\to Y_1, \qquad \Lambda:Y_1\to Y_2
\]
and a discrete set
\[
\Sigma\subset \Omega_a\setminus \overline{\mathbb{D}}
\]
such that
\[
E_+(z)^{-1}E_-(z)=\Delta\,F_+(z)^{-1}F_-(z)\,\Lambda,
\qquad z\in \Omega_a\setminus \Sigma.
\]
\end{enumerate}

Here the set $\Sigma$ is chosen so that $E_+$ and $F_+$ are invertible on $\Omega_a\setminus \Sigma$. In view of the discussion in Section~2, the possible singularities of $E_+^{-1}E_-$ and $F_+^{-1}F_-$ in $\mathbb{D}\cup (\Omega_a\cap \mathbb{T})$ are removable, and hence the only possible non-removable singularities lie in a discrete subset of $\Omega_a\setminus \overline{\mathbb{D}}$.
\end{thm}

\begin{proof}
For $j=1,2$, let $\Theta_j$ denote the Sz.-Nagy--Foias characteristic function of $T_j$. Also let
\[
K_w^{\mathfrak{E}}(z)\qquad \text{and} \qquad K_w^{\mathfrak{F}}(z)
\]
denote the reproducing kernels of $\mathcal{B}(\mathfrak{E})$ and $\mathcal{B}(\mathfrak{F})$, respectively.

By Theorem~\ref{Theorem M} and equation \eqref{T}, we have
\[
T_1 \cong \mathbf{P}_{\mathcal{D}(\mathfrak{T}_1)}\mathfrak{T}_1|_{\mathcal{D}(\mathfrak{T}_1)},
\qquad
T_2 \cong \mathbf{P}_{\mathcal{D}(\mathfrak{T}_2)}\mathfrak{T}_2|_{\mathcal{D}(\mathfrak{T}_2)}.
\]
Therefore, it is enough to compare the compressed multiplication operators.

\medskip

\noindent
$(1)\Rightarrow(2)$.
Assume that $T_1$ and $T_2$ are unitarily equivalent. Since the characteristic functions of unitarily equivalent contractions coincide, there exist unitary operators
\[
U:\mathcal D_{T_2}\to \mathcal D_{T_1},
\qquad
U_*:\mathcal D_{T_2^*}\to \mathcal D_{T_1^*}
\]
such that
\[
\Theta_1(z)=U_*\,\Theta_2(z)\,U^{-1},
\qquad z\in \mathbb D.
\]

Now, by Theorem~\ref{Theorem 5.1}, for each $j=1,2$, the characteristic function $\Theta_j(z)$
coincides on $\mathbb D$ with the operator valued function
\[
P_{Y_j}(z)\big|_{\{0\}\oplus \mathcal D_{T_j}}.
\]
Therefore the operator valued functions
\[
P_{Y_1}(z)\big|_{\{0\}\oplus \mathcal D_{T_1}}
\qquad\text{and}\qquad
P_{Y_2}(z)\big|_{\{0\}\oplus \mathcal D_{T_2}}
\]
also coincide on $\mathbb D$.

Further, by Remark~\ref{Remark 5.2}, the reproducing kernels of $\mathcal B(\mathfrak E)$ and
$\mathcal B(\mathfrak F)$ satisfy
\[
K_w^{\mathfrak E}(z)=P_{Y_1}(z)P_{Y_1}(w)^*,
\qquad
K_w^{\mathfrak F}(z)=P_{Y_2}(z)P_{Y_2}(w)^*,
\qquad z,w\in \Omega_a.
\]
Hence the coincidence of the above operator valued functions implies that there exists
a unitary operator
\[
\Xi:Y_2\to Y_1
\]
such that
\begin{equation}\label{eqn5.10}
K_w^{\mathfrak E}(z)=\Xi\,K_w^{\mathfrak F}(z)\,\Xi^*,
\qquad z,w\in \mathbb D.
\end{equation}
Equivalently, by Remark~\ref{Remark 5.3}, this is precisely the kernel form of the coincidence of
the characteristic functions.

We now use the formulas from Theorem~\ref{Theorem C}. Since
\[
E_+(z)=\rho_\beta(z)\,K_\beta^{\mathfrak E}(z)
\bigl(\rho_\beta(\beta)K_\beta^{\mathfrak E}(\beta)\bigr)^{-1/2},
\]
\[
F_+(z)=\rho_\beta(z)\,K_\beta^{\mathfrak F}(z)
\bigl(\rho_\beta(\beta)K_\beta^{\mathfrak F}(\beta)\bigr)^{-1/2},
\]
it follows from \eqref{eqn5.10} that
\[
E_+(z)=\Xi\,F_+(z)\,\Xi^*,
\qquad z\in \mathbb D.
\]
Similarly, using
\[
E_-(z)= -\,\rho_{1/\overline{\beta}}(z)\,K_{1/\overline{\beta}}^{\mathfrak E}(z)\,
\Bigl(-\rho_{1/\overline{\beta}}(1/\overline{\beta})\,K_{1/\overline{\beta}}^{\mathfrak E}(1/\overline{\beta})\Bigr)^{-1/2},
\]
\[
F_-(z)= -\,\rho_{1/\overline{\beta}}(z)\,K_{1/\overline{\beta}}^{\mathfrak F}(z)\,
\Bigl(-\rho_{1/\overline{\beta}}(1/\overline{\beta})\,K_{1/\overline{\beta}}^{\mathfrak F}(1/\overline{\beta})\Bigr)^{-1/2}.
\]
we obtain
\[
E_-(z)=\Xi\,F_-(z)\,\Xi^*,
\qquad z\in \mathbb D.
\]
Therefore,
\begin{equation}\label{eqn5.11}
E_+(z)^{-1}E_-(z)
=
\Xi\,F_+(z)^{-1}F_-(z)\,\Xi^*,
\qquad z\in \mathbb D.
\end{equation}

By the discussion in Section~\ref{Section 2}, the operator valued functions $E_+^{-1}E_-$ and
$F_+^{-1}F_-$ are analytic on $\mathbb D$, extend analytically across
$\Omega_a\cap \mathbb T$ up to removable singularities, and are meromorphic on
$\Omega_a$, with possible non-removable singularities only at a discrete subset of
$\Omega_a\setminus \overline{\mathbb D}$. Let $\Sigma\subset \Omega_a\setminus
\overline{\mathbb D}$ contain all such singularities of both functions. Then both sides
of \eqref{eqn5.11} are analytic on $\Omega_a\setminus \Sigma$, and since they agree on
$\mathbb D$, the identity theorem yields
\[
E_+(z)^{-1}E_-(z)
=
\Xi\,F_+(z)^{-1}F_-(z)\,\Xi^*,
\qquad z\in \Omega_a\setminus \Sigma.
\]
Thus assertion $(2)$ holds with
\[
\Delta=\Xi,
\qquad
\Lambda=\Xi^*.
\]
\medskip

\noindent
$(2)\Rightarrow(1)$.
Assume that there exist unitary operators
\[
\Delta:Y_2\to Y_1,
\qquad
\Lambda:Y_1\to Y_2
\]
and a discrete set $\Sigma\subset \Omega_a\setminus \overline{\mathbb{D}}$ such that
\begin{equation}\label{eq5.10}
E_+(z)^{-1}E_-(z)=\Delta\,F_+(z)^{-1}F_-(z)\,\Lambda,
\qquad z\in \Omega_a\setminus \Sigma.
\end{equation}
Set
\[
\Omega:=\Omega_a\setminus \Sigma.
\]
By construction, both $E_+$ and $F_+$ are invertible on $\Omega$.

By Theorem~\ref{Theorem C},
\[
K_w^{\mathfrak{E}}(z)
=
E_+(z)\,
\frac{I-\bigl(E_+(z)^{-1}E_-(z)\bigr)\bigl(E_+(w)^{-1}E_-(w)\bigr)^*}{\rho_w(z)}
\,E_+(w)^*,
\]
\[
K_w^{\mathfrak{F}}(z)
=
F_+(z)\,
\frac{I-\bigl(F_+(z)^{-1}F_-(z)\bigr)\bigl(F_+(w)^{-1}F_-(w)\bigr)^*}{\rho_w(z)}
\,F_+(w)^*,
\qquad z,w\in \Omega.
\]
Using \eqref{eq5.10} and the fact that $\Lambda$ is unitary, we get
\[
\bigl(E_+^{-1}E_-\bigr)(z)\bigl(E_+^{-1}E_-\bigr)(w)^*
=
\Delta\,
\bigl(F_+^{-1}F_-\bigr)(z)\bigl(F_+^{-1}F_-\bigr)(w)^*
\,\Delta^*.
\]
Hence
\[
I-\bigl(E_+^{-1}E_-\bigr)(z)\bigl(E_+^{-1}E_-\bigr)(w)^*
=
\Delta
\Bigl(
I-\bigl(F_+^{-1}F_-\bigr)(z)\bigl(F_+^{-1}F_-\bigr)(w)^*
\Bigr)
\Delta^*.
\]
Therefore, if we define
\[
W(z):=E_+(z)\,\Delta\,F_+(z)^{-1},
\qquad z\in \Omega,
\]
then
\[
K_w^{\mathfrak{E}}(z)=W(z)\,K_w^{\mathfrak{F}}(z)\,W(w)^*,
\qquad z,w\in \Omega.
\]

Now let
\[
\mathcal{L}_{\mathfrak{F}}
:=
\operatorname{span}\left\{
K_w^{\mathfrak{F}}(\cdot)u:\ w\in \Omega,\ u\in Y_2
\right\}.
\]
Since $\Omega$ is obtained from $\Omega_a$ by deleting only a discrete set, it has an accumulation point in $\Omega_a$. Hence $\mathcal{L}_{\mathfrak{F}}$ is dense in $\mathcal{B}(\mathfrak{F})$.

Define a linear map
\[
\mathcal{U}:\mathcal{L}_{\mathfrak{F}}\to \mathcal{B}(\mathfrak{E})
\]
by
\[
\mathcal{U}
\left(
\sum_{j=1}^n K_{w_j}^{\mathfrak{F}}(\cdot)u_j
\right)
:=
\sum_{j=1}^n
K_{w_j}^{\mathfrak{E}}(\cdot)\,
\bigl(W(w_j)^*\bigr)^{-1}u_j.
\]
From \eqref{eq5.10}, for every $z\in \Omega$,
\begin{equation}\label{ke}
K_{w_j}^{\mathfrak{E}}(z)\bigl(W(w_j)^*\bigr)^{-1}
=
W(z)K_{w_j}^{\mathfrak{F}}(z).
\end{equation}
Therefore, if $f\in \mathcal{L}_{\mathfrak{F}}$, then
\begin{equation}\label{5.11}
(\mathcal{U}f)(z)=W(z)f(z),
\qquad z\in \Omega.
\end{equation}

We claim that $\mathcal{U}$ is an isometry on $\mathcal{L}_{\mathfrak{F}}$. Let
\[
f=\sum_{i=1}^n K_{w_i}^{\mathfrak{F}}(\cdot)u_i,
\qquad
g=\sum_{j=1}^m K_{\lambda_j}^{\mathfrak{F}}(\cdot)v_j.
\]
Using the reproducing property and \eqref{ke}, we obtain
\[
\begin{aligned}
\langle \mathcal{U}f,\mathcal{U}g\rangle_{\mathcal{B}(\mathfrak{E})}
&=
\sum_{i,j}
\left\langle
K_{w_i}^{\mathfrak{E}}(\lambda_j)\bigl(W(w_i)^*\bigr)^{-1}u_i,\,
\bigl(W(\lambda_j)^*\bigr)^{-1}v_j
\right\rangle \\
&=
\sum_{i,j}
\langle K_{w_i}^{\mathfrak{F}}(\lambda_j)u_i,v_j\rangle \\
&=
\langle f,g\rangle_{\mathcal{B}(\mathfrak{F})}.
\end{aligned}
\]
Hence $\mathcal{U}$ extends uniquely to a unitary operator
\[
\mathcal{U}:\mathcal{B}(\mathfrak{F})\to \mathcal{B}(\mathfrak{E}).
\]

We next show that $\mathcal{U}$ intertwines $\mathfrak{T}_2$ and $\mathfrak{T}_1$. Let
$f\in \mathcal{D}(\mathfrak{T}_2)$. Then $zf\in \mathcal{B}(\mathfrak{F})$. From \eqref{5.11},
\[
(\mathcal{U}(\mathfrak{T}_2 f))(z)
=
W(z)\,z f(z)
=
z\,W(z)f(z)
=
z\,(\mathcal{U}f)(z),
\qquad z\in \Omega.
\]
Thus $z(\mathcal{U}f)\in \mathcal{B}(\mathfrak{E})$, and hence $\mathcal{U}f\in \mathcal{D}(\mathfrak{T}_1)$, with
\[
\mathcal{U}\mathfrak{T}_2f=\mathfrak{T}_1\mathcal{U}f.
\]
Applying the same argument to $\mathcal{U}^{-1}$, we obtain
\[
\mathcal{U}\mathcal{D}(\mathfrak{T}_2)=\mathcal{D}(\mathfrak{T}_1).
\]
Since $\mathcal{U}$ is unitary and maps $\mathcal{D}(\mathfrak{T}_2)$ onto $\mathcal{D}(\mathfrak{T}_1)$, we have
\[
\mathcal{U}\mathbf{P}_{\mathcal{D}(\mathfrak{T}_2)}=
\mathbf{P}_{\mathcal{D}(\mathfrak{T}_1)}\mathcal{U}.
\]
Therefore,
\[
\mathcal{U}
\Bigl(
\mathbf{P}_{\mathcal{D}(\mathfrak{T}_2)}\mathfrak{T}_2|_{\mathcal{D}(\mathfrak{T}_2)}
\Bigr)
=
\Bigl(
\mathbf{P}_{\mathcal{D}(\mathfrak{T}_1)}\mathfrak{T}_1|_{\mathcal{D}(\mathfrak{T}_1)}
\Bigr)
\mathcal{U}.
\]
Thus
\[
\mathbf{P}_{\mathcal{D}(\mathfrak{T}_1)}\mathfrak{T}_1|_{\mathcal{D}(\mathfrak{T}_1)}
\qquad\text{and}\qquad
\mathbf{P}_{\mathcal{D}(\mathfrak{T}_2)}\mathfrak{T}_2|_{\mathcal{D}(\mathfrak{T}_2)}
\]
are unitarily equivalent. Since
\[
T_j \cong \mathbf{P}_{\mathcal{D}(\mathfrak{T}_j)}\mathfrak{T}_j|_{\mathcal{D}(\mathfrak{T}_j)},
\qquad j=1,2,
\]
it follows that $T_1$ and $T_2$ are unitarily equivalent. This proves $(2)\Rightarrow(1)$.

The proof is complete.
\end{proof}

\section{The canonical contraction in the de Branges model and the $L^2$ realization}
\label{Section 6}

In this section, we study the quotient
\[
        \Theta_\mathfrak E=E_+^{-1}E_-
\]
associated with the de Branges operator $\mathfrak{E}=(E_-, E_+)$ obtained in Theorem \ref{Theorem C}.
We first prove that this quotient is purely contractive. Then, under the additional
assumption that the de Branges operator is entire, we construct a canonical contraction \(S_\mathfrak E\) on the corresponding de Branges space
and show that $\Theta_{S_\mathfrak E}$ coincides with $\Theta_\mathfrak E=E_+^{-1}E_- $.
We also describe the larger \(L^2\)-ambient space in which this model fits into a unit-disc analogue of the Arov--Dym scattering picture of a $\mathcal{B}(\mathfrak{E})$ space with respect to a de Branges matrix $\mathfrak{E}= (E_-, E_+)$, which is based on a matrix valued reproducing kernel, and has been worked out in Section $8$ of \cite{ArDymone}.

\textbf{I.~Pure contractivity of the quotient \(E_+^{-1}E_-\):}

Let \(T\) be a cnu contraction satisfying the hypotheses of
Theorem~\ref{Theorem M}. Let
\[
       \mathfrak E=(E_-,E_+)
\]
be the de Branges operator associated with \(T\) which is obtained through Theorem~\ref{Theorem C}.
We define
\[
        \Theta_\mathfrak E(z)=E_+(z)^{-1}E_-(z),\qquad z\in\mathbb D .
\]
Here, as usual, the possible removable singularities of \(E_+^{-1}E_-\) are
understood to be removed. By the definition of a de Branges operator, the quotient \(E_+^{-1}E_-\) is
contractive on the unit disc. Hence
\[
        I-\Theta_\mathfrak E(z)^*\Theta_\mathfrak E(z)\geq 0,
        \qquad z\in\mathbb D .
\]
Therefore \(\Theta_\mathfrak E\) belongs to the Schur class. We now prove that \(\Theta_\mathfrak E\) is purely contractive. Recall that an operator valued
Schur function \(\Theta\) is called purely contractive if
\[
        \|\Theta(0)x\|<\|x\|
        \qquad \text{for every } x \neq 0 .
\]

\begin{lemma}
Let \(\mathfrak E=(E_-,E_+)\) be the de Branges operator as obtained in Theorem \ref{Theorem C}. Then
\[
\Theta_\mathfrak E(z)=E_+(z)^{-1}E_-(z), \qquad z\in \mathbb D,
\]
is purely contractive.
\end{lemma}
\begin{proof}
By Theorem \ref{Theorem C}, the functions \(E_+\) and \(E_-\) are given by
\[
E_+(z)
=
\rho_\beta(z)K^\mathfrak E_\beta(z)
\left(\rho_\beta(\beta)K^\mathfrak E_\beta(\beta)\right)^{-1/2},
\]
and
\[
E_-(z)
=
-\rho_{1/\overline{\beta}}(z)K^\mathfrak E_{1/\overline{\beta}}(z)
\left(
-\rho_{1/\overline{\beta}}(1/\overline{\beta})
K^\mathfrak E_{1/\overline{\beta}}(1/\overline{\beta})
\right)^{-1/2}.
\]

In the proof of Theorem \ref{Theorem M}, it is verified that \(K^\mathfrak E_\beta(\beta)\) is invertible.
Since
\[
\rho_\beta(\beta)=1-|\beta|^2>0,
\]
the formula for \(E_+(\beta)\) implies that \(E_+(\beta)\) is invertible.

The reproducing kernel corresponding to the de Branges operator \(\mathfrak E\) is
\[
K^\mathfrak E_w(z)
=
\frac{
E_+(z)E_+(w)^*
-
E_-(z)E_-(w)^*
}{
1-z\overline w
}.
\]
Since
\[
E_-(z)=E_+(z)\Theta_\mathfrak E(z),
\]
we obtain
\[
K^\mathfrak E_w(z)
=
E_+(z)
\frac{
I-\Theta_\mathfrak E(z)\Theta_\mathfrak E(w)^*
}{
1-z\overline w
}
E_+(w)^*.
\]
Taking \(z=w=\beta\), we get
\[
K^\mathfrak E_\beta(\beta)
=
\frac{
E_+(\beta)
\left(I-\Theta_\mathfrak E(\beta)\Theta_\mathfrak E(\beta)^*\right)
E_+(\beta)^*
}{
1-|\beta|^2
}.
\]
Equivalently,
\[
I-\Theta_\mathfrak E(\beta)\Theta_\mathfrak E(\beta)^*
=
(1-|\beta|^2)
E_+(\beta)^{-1}
K^\mathfrak E_\beta(\beta)
E_+(\beta)^{-*}.
\]
Here
\[
E_+(\beta)^{-*}:=\left(E_+(\beta)^{-1}\right)^*.
\]
Since \(K^\mathfrak E_\beta(\beta)\) and \(E_+(\beta)\) are invertible, it follows that
\[
I-\Theta_\mathfrak E(\beta)\Theta_\mathfrak E(\beta)^*
\]
is positive and invertible. Therefore there exists a constant \(c>0\) such that
\[
I-\Theta_\mathfrak E(\beta)\Theta_\mathfrak E(\beta)^*\geq cI.
\]
Hence
\[
\Theta_\mathfrak E(\beta)\Theta_\mathfrak E(\beta)^*\leq (1-c)I.
\]
Consequently,
\[
\|\Theta_\mathfrak E(\beta)\|<1.
\]

We now show that \(\Theta_\mathfrak E\) is purely contractive. Suppose, on the contrary, that
\(\Theta_\mathfrak E\) is not purely contractive. Then there exists a non-zero vector \(x\) such
that
\[
\|\Theta_\mathfrak E(0)x\|=\|x\|.
\]
Since \(\Theta_\mathfrak E\) is contractive on \(\mathbb D\), we have
\[
\|\Theta_\mathfrak E(z)x\|\leq \|x\|, \qquad z\in \mathbb D.
\]
Define the scalar analytic function
\[
\varphi(z)
=
\frac{\langle \Theta_\mathfrak E(z)x,\Theta_\mathfrak E(0)x\rangle}{\|x\|^2},
\qquad z\in \mathbb D.
\]
Then
\[
|\varphi(z)|
\leq
\frac{\|\Theta_\mathfrak E(z)x\|\,\|\Theta_\mathfrak E(0)x\|}{\|x\|^2}
\leq 1,
\qquad z\in \mathbb D.
\]
Moreover,
\[
\varphi(0)
=
\frac{\langle \Theta_\mathfrak E(0)x,\Theta_\mathfrak E(0)x\rangle}{\|x\|^2}
=
1.
\]
By the maximum modulus principle, \(\varphi\equiv 1\) on \(\mathbb D\). Therefore
equality holds in the Cauchy--Schwarz inequality for every \(z\in\mathbb D\), and
we get
\[
\Theta_\mathfrak E(z)x=\Theta_\mathfrak E(0)x, \qquad z\in \mathbb D.
\]
In particular,
\[
\Theta_\mathfrak E(\beta)x=\Theta_\mathfrak E(0)x.
\]
Thus
\[
\|\Theta_\mathfrak E(\beta)x\|=\|\Theta_\mathfrak E(0)x\|=\|x\|.
\]
This contradicts the strict inequality
\[
\|\Theta_\mathfrak E(\beta)\|<1.
\]
Therefore no such non-zero vector \(x\) exists. Hence
\[
\|\Theta_\mathfrak E(0)x\|<\|x\|,
\qquad x \neq 0.
\]
Thus \(\Theta_\mathfrak E=E_+^{-1}E_-\) is purely contractive.
\end{proof}

\textbf{II.~The Hardy-space model associated with \(\mathfrak E\):}

We now describe the canonical Hardy-space model determined by the de Branges
operator \(\mathfrak E=(E_-,E_+)\). Throughout this subsection and later, we assume that \(\mathfrak E\) is
entire, that is,
\[
        \Omega=\mathbb C .
\]
Consequently,
\[
        \Omega\cap\mathbb T=\mathbb T .
\]
 The boundary values of \(\Theta_\mathfrak E\) satisfy
\[
        \Theta_\mathfrak E(\zeta)^*\Theta_\mathfrak E(\zeta)=I,
        \qquad
        \Theta_\mathfrak E(\zeta)\Theta_\mathfrak E(\zeta)^*=I,
\]
for all \(\zeta\in\mathbb T\). Hence, by the previous subsection, \(\Theta_\mathfrak E\) is a purely contractive inner Schur
function.

Let
\[
       \mathcal B_{\mathbb D}(\mathfrak E):=
        \{f\vert_{\mathbb D}: f\in \mathcal B(\mathfrak E)\}.
\]
Since the functions in \(\mathcal B(\mathfrak E)\) are analytic on \(\mathbb C\), the restriction
map
\[
        R_{\mathbb D}:\mathcal B(\mathfrak E)\longrightarrow \mathcal B_{\mathbb D}(\mathfrak E),
        \qquad
        R_{\mathbb D}f=f\vert_{\mathbb D},
\]
is injective by the identity theorem. We give \(\mathcal B_{\mathbb D}(\mathfrak E)\) the
transported norm
\[
        \|f\vert_{\mathbb D}\|_{\mathcal B_{\mathbb D}(\mathfrak E)}
        :=
        \|f\|_{\mathcal B(\mathfrak E)} .
\]
With this norm, \(R_{\mathbb D}\) is a unitary map from \(\mathcal B(\mathfrak E)\) onto
\(\mathcal B_{\mathbb D}(\mathfrak E)\). Thus \(\mathcal B(\mathfrak E)\) and \(\mathcal B_{\mathbb D}(\mathfrak E)\) are not different
Hilbert spaces for our purposes; \(\mathcal B_{\mathbb D}(\mathfrak E)\) is only the unit-disc
realization of \(\mathcal B(\mathfrak E)\).

By the Hardy-space characterization of the de Branges space (see Theorem \ref{Theorem HSC}), multiplication
by \(E_+\) gives a unitary identification between \(\mathcal B_{\mathbb D}(\mathfrak E)\) and the
de Branges--Rovnyak space \(\mathcal H(\Theta_\mathfrak E)\). Hence
\[
        \mathcal B_{\mathbb D}(\mathfrak E)=E_+\mathcal H(\Theta_\mathfrak E).
\]
Since \(\Theta_\mathfrak E\) is inner,
\[
        \mathcal H(\Theta_\mathfrak E)
        =
        K_{\Theta_\mathfrak E}
        :=
        H_Y^2(\mathbb D)\ominus \Theta_\mathfrak E H_Y^2(\mathbb D),
\]
where $Y$ is the coefficient space as in Theorem \ref{Theorem C}. 
Therefore
\[
        \mathcal B_{\mathbb D}(\mathfrak E)=E_+K_{\Theta_\mathfrak E}.
\]

For the purpose of defining the canonical compression, we introduce the following space
\[
        H_\mathfrak E^2:=E_+H_Y^2(\mathbb D).
\]
Thus
\[
        H_\mathfrak E^2=
        \{E_+g:g\in H_Y^2(\mathbb D)\},
\]
with norm
\[
        \|E_+g\|_{H_\mathfrak E^2}:=\|g\|_{H_Y^2(\mathbb D)}.
\]
The map
\[
        W_\mathfrak E:H_\mathfrak E^2\longrightarrow H_Y^2(\mathbb D),
        \qquad
        W_\mathfrak E(E_+g)=g,
\]
is unitary. Since
\[
        \mathcal B_{\mathbb D}(\mathfrak E)=E_+K_{\Theta_\mathfrak E},
\]
we see that \(\mathcal B_{\mathbb D}(\mathfrak E)\) is a closed subspace of \(H_\mathfrak E^2\).

Let
\[
        M_z:H_\mathfrak E^2\longrightarrow H_\mathfrak E^2
\]
be multiplication operator by \(z\). This is well-defined because
\[
        M_z(E_+g)=zE_+g=E_+(zg),
\]
and \(zg\in H_Y^2(\mathbb D)\). Moreover, under the unitary map \(W_\mathfrak E\), the
operator \(M_z\) is unitarily equivalent to the unilateral shift on
\(H_Y^2(\mathbb D)\). Hence \(M_z\) is an isometry on \(H_\mathfrak E^2\).

Let
\[
        \mathbf P_{\mathcal B_{\mathbb D}(\mathfrak E)}:H_\mathfrak E^2\longrightarrow \mathcal B_{\mathbb D}(\mathfrak E)
\]
denote the orthogonal projection. We now define the canonical compression
associated with \(\mathfrak E\) by
\[
        S_\mathfrak E:=\mathbf P_{\mathcal B_{\mathbb D}(\mathfrak E)}M_z\vert_{\mathcal B_{\mathbb D}(\mathfrak E)} .
\]
Equivalently,
\[
        S_\mathfrak E f=\mathbf P_{\mathcal B_{\mathbb D}(\mathfrak E)}(zf),
        \qquad f\in \mathcal B_{\mathbb D}(\mathfrak E).
\]
Since \(S_\mathfrak E\) is the compression of an isometry, \(S_\mathfrak E\) is a contraction.

Under the unitary map \(W_\mathfrak E\), the subspace \(\mathcal B_{\mathbb D}(\mathfrak E)=E_+K_{\Theta_\mathfrak E}\)
is mapped onto \(K_{\Theta_\mathfrak E}\). Therefore
\[
        W_\mathfrak ES_\mathfrak EW_\mathfrak E^{-1}
        =
        \mathbf P_{K_{\Theta_\mathfrak E}}M_z\vert_{K_{\Theta_\mathfrak E}} .
\]
Let
\[
        S_{\Theta_\mathfrak E}:=
        \mathbf P_{K_{\Theta_\mathfrak E}}M_z\vert_{K_{\Theta_\mathfrak E}} .
\]
Then
\[
        S_\mathfrak E\cong S_{\Theta_\mathfrak E}.
\]
The operator \(S_{\Theta_\mathfrak E}\) is the standard Sz.-Nagy--Foias model operator
corresponding to the purely contractive inner Schur function \(\Theta_\mathfrak E\) (see \cite[Chapter VI, Theorem 3.1]{Nagy-Foais}).
Consequently, its characteristic function coincides with \(\Theta_\mathfrak E\). Hence $\Theta_{S_\mathfrak E}$ coincides with $\Theta_\mathfrak E=E_+^{-1}E_-$.
Thus the de Branges operator \(\mathfrak E=(E_-,E_+)\) canonically determines a
contraction \(S_\mathfrak E\) whose Sz.-Nagy--Foias characteristic function coincides with the
operator valued function \(E_+^{-1}E_-\).

The discussions hereafter are motivated by Section $8$ of Arov and Dym \cite{ArDymone}. Though the methods are same, we include these discussions in the disc set up for the sake of completeness.

\textbf{III.~The \(L^2\)-ambient space:}

The space \(H_\mathfrak E^2=E_+H_Y^2(\mathbb D)\) is the correct ambient
space for the canonical model compression \(S_\mathfrak E\). However, in order to make
the connection with the Arov--Dym construction and the Lax--Phillips
scattering picture \cite{LP}, one has to pass to a larger \(L^2\)-space.

Let \(L_Y^2(\mathbb T)\) be the usual \(Y\)-valued \(L^2\)-space on the unit
circle. We define
\[
        L_\mathfrak E^2(\mathbb T):=E_+L_Y^2(\mathbb T).
\]
That is,
\[
        L_\mathfrak E^2(\mathbb T)
        =
        \{E_+g:g\in L_Y^2(\mathbb T)\}.
\]
We equip this space with the transported norm
\[
        \|E_+g\|_{L_\mathfrak E^2}:=\|g\|_{L_Y^2(\mathbb T)}.
\]
With this norm, the map
\[
        M_{E_+}:L_Y^2(\mathbb T)\longrightarrow L_\mathfrak E^2(\mathbb T),
        \qquad
        M_{E_+}g=E_+g,
\]
is unitary.

We note that this definition avoids any difficulty caused by possible
non-invertibility of \(E_+\) on the boundary. Indeed, since \(E_+\) is
Fredholm-valued and invertible at least at one point, the analytic Fredholm
theorem implies that the set
\[
        \Sigma_+:=\{\lambda\in\mathbb C:E_+(\lambda)
        \text{ is not invertible}\}
\]
is discrete. Hence \(\Sigma_+\cap\mathbb T\) is discrete. Therefore \(E_+^{-1}\)
is defined a.e. on \(\mathbb T\), and the values on \(\Sigma_+\cap\mathbb T\)
are irrelevant for \(L^2\)-purposes.

Since \(E_+(\zeta)\) is invertible for a.e. \(\zeta\in\mathbb T\), then the above
transported norm can be written in weighted form. Namely, put
\[
        \Delta_\mathfrak E(\zeta)
        =
        \big(E_+(\zeta)E_+(\zeta)^*\big)^{-1}.
\]
Then, for \(f=E_+g\),
\[
\begin{aligned}
        \|f\|_{L_\mathfrak E^2}^2
        &=
        \|g\|_{L_Y^2(\mathbb T)}^2                                            \\
        &=
        \int_{\mathbb T}\|g(\zeta)\|_Y^2\,dm(\zeta)                            \\
        &=
        \int_{\mathbb T}
        \left\langle
        \big(E_+(\zeta)E_+(\zeta)^*\big)^{-1}f(\zeta),
        f(\zeta)
        \right\rangle_Y\,dm(\zeta)                                             \\
        &=
        \int_{\mathbb T}
        \langle \Delta_\mathfrak E(\zeta)f(\zeta),f(\zeta)\rangle_Y\,dm(\zeta).
\end{aligned}
\]

The space \(H_\mathfrak E^2\) introduced above is precisely the positive Hardy part of
\(L_\mathfrak E^2(\mathbb T)\):
\[
        H_\mathfrak E^2=E_+H_Y^2(\mathbb D)\subset L_\mathfrak E^2(\mathbb T).
\]
Moreover,
\[
        \mathcal B_{\mathbb D}(\mathfrak E)=E_+K_{\Theta_\mathfrak E}\subset H_\mathfrak E^2\subset L_\mathfrak E^2(\mathbb T).
\]
Therefore the correct hierarchy of spaces is
\[
        \mathcal B_{\mathbb D}(\mathfrak E)\subset H_\mathfrak E^2\subset L_\mathfrak E^2(\mathbb T).
\]

Let
\[
        H_{Y,-}^2:=L_Y^2(\mathbb T)\ominus H_Y^2(\mathbb D)
\]
be the negative Hardy space. Since \(\Theta_\mathfrak E\) is inner,
\[
        H_Y^2(\mathbb D)
        =
        K_{\Theta_\mathfrak E}\oplus \Theta_\mathfrak EH_Y^2(\mathbb D).
\]
Also,
\[
        L_Y^2(\mathbb T)
        =
        H_{Y,-}^2\oplus H_Y^2(\mathbb D).
\]
Consequently,
\[
        L_Y^2(\mathbb T)
        =
        H_{Y,-}^2
        \oplus
        K_{\Theta_\mathfrak E}
        \oplus
        \Theta_\mathfrak EH_Y^2(\mathbb D).
\]
Transporting this orthogonal decomposition by the unitary map \(M_{E_+}\), we
obtain
\[
        L_\mathfrak E^2(\mathbb T)
        =
        E_+H_{Y,-}^2
        \oplus
        E_+K_{\Theta_\mathfrak E}
        \oplus
        E_+\Theta_\mathfrak EH_Y^2(\mathbb D).
\]
Since
\[
        E_+K_{\Theta_\mathfrak E}=\mathcal B_{\mathbb D}(\mathfrak E)
\]
and
\[
        E_+\Theta_\mathfrak EH_Y^2(\mathbb D)
        =
        E_-H_Y^2(\mathbb D),
\]
we get the decomposition
\[
        L_\mathfrak E^2(\mathbb T)
        =
        E_+H_{Y,-}^2
        \oplus
        \mathcal B_{\mathbb D}(\mathfrak E)
        \oplus
        E_-H_Y^2(\mathbb D).
\]
Equivalently, if we put
\[
        \mathcal D_-^\mathfrak E:=E_+H_{Y,-}^2
\]
and
\[
        \mathcal D_+^\mathfrak E:=E_-H_Y^2(\mathbb D),
\]
then
\[
        L_\mathfrak E^2(\mathbb T)
        =
        \mathcal D_-^\mathfrak E
        \oplus
       \mathcal B_{\mathbb D}(\mathfrak E)
        \oplus
        \mathcal D_+^\mathfrak E.
\]
This is the unit-disc analogue of the Arov--Dym decomposition.

\textbf{IV.~The Lax-Phillips scattering operator and the unitary dilation picture:}

The two unitary maps
\[
        \mathcal F_-:L_\mathfrak E^2(\mathbb T)\longrightarrow L_Y^2(\mathbb T),
        \qquad
        \mathcal F_-f=E_+^{-1}f,
\]
and
\[
        \mathcal F_+:L_\mathfrak E^2(\mathbb T)\longrightarrow L_Y^2(\mathbb T),
        \qquad
        \mathcal F_+f=E_-^{-1}f,
\]
are the unit-disc analogues of the generalized Fourier transforms associated with a Lax-Phillips scheme. For the basic definitions of Lax-Phillips scheme, we refer to \cite{LP} and \cite[Section 8]{ArDymone}.

Indeed, if \(f=E_+g\), then
\[
        \mathcal F_-f=g.
\]
Also, since
\[
        E_-=E_+\Theta_\mathfrak E,
\]
we have
\[
        f=E_+g=E_-\Theta_\mathfrak E^*g
\]
a.e. on \(\mathbb T\). Hence
\[
        \mathcal F_+f=\Theta_\mathfrak E^*g.
\]
Therefore
\[
        \mathcal F_+=M_{\Theta_\mathfrak E}^*\mathcal F_-,
\]
where \(M_{\Theta_\mathfrak E}\) denotes multiplication by \(\Theta_\mathfrak E\) on
\(L_Y^2(\mathbb T)\).

The corresponding scattering operator is
\[
        \mathcal F_-\mathcal F_+^*.
\]
For \(g\in L_Y^2(\mathbb T)\), we have
\[
        \mathcal F_+^*g=E_-g.
\]
Therefore
\[
        \mathcal F_-\mathcal F_+^*g
        =
        E_+^{-1}E_-g
        =
        \Theta_\mathfrak Eg.
\]
Hence
\[
        \mathcal F_-\mathcal F_+^*
        =
        M_{\Theta_\mathfrak E}.
\]
Thus the scattering operator is precisely
\[
        \Theta_\mathfrak E=E_+^{-1}E_-.
\]

This shows that the same quotient \(E_+^{-1}E_-\) plays two roles: it is the
analytic function defining the de Branges kernel, and it is also the scattering
operator associated with a Lax-Phillips scheme.

Define
\[
        U_\mathfrak E:L_\mathfrak E^2(\mathbb T)\longrightarrow L_\mathfrak E^2(\mathbb T)
\]
by
\[
        (U_\mathfrak Ef)(\zeta)=\zeta f(\zeta),
        \qquad \zeta\in\mathbb T.
\]
Equivalently, if \(f=E_+g\), then
\[
        U_\mathfrak E(E_+g)=E_+(\zeta g).
\]
Under the unitary map \(\mathcal F_-\), the operator \(U_\mathfrak E\) becomes ordinary
multiplication by \(\zeta\) on \(L_Y^2(\mathbb T)\):
\[
        \mathcal F_-U_\mathfrak E\mathcal F_-^{-1}=M_\zeta .
\]
Thus \(U_\mathfrak E\) is unitary.

The subspaces
\[
        \mathcal D_-^\mathfrak E=E_+H_{Y,-}^2
\]
and
\[
        \mathcal D_+^\mathfrak E=E_-H_Y^2(\mathbb D)
\]
satisfy the standard Lax--Phillips invariance relations
\[
        U_\mathfrak E^{-1}\mathcal D_-^\mathfrak E\subseteq \mathcal D_-^\mathfrak E
\]
and
\[
        U_\mathfrak E\mathcal D_+^\mathfrak E\subseteq \mathcal D_+^\mathfrak E.
\]
Indeed,
\[
        \zeta^{-1}H_{Y,-}^2\subseteq H_{Y,-}^2
\]
and
\[
        \zeta H_Y^2(\mathbb D)\subseteq H_Y^2(\mathbb D).
\]

Moreover, the compression of \(U_\mathfrak E\) to the central space
\(\mathcal B_{\mathbb D}(\mathfrak E)\) is the same canonical contraction considered above:
\[
        S_\mathfrak E=\mathbf P_{\mathcal B_{\mathbb D}(\mathfrak E)}U_\mathfrak E\vert_{\mathcal B_{\mathbb D}(\mathfrak E)}.
\]
Since \(U_\mathfrak E\) is a unitary operator on \(L_\mathfrak E^2(\mathbb T)\), this gives a
unitary dilation picture for the contraction \(S_\mathfrak E\).

\bigskip
\textbf{Acknowledgements:} 
This work is partially supported by the FIST program of the Department of Science and Technology, Government of India, Reference No. SR/FST/MS-I/2018/22(C). The research of the second author is supported by the MATRICS grant of SERB
(MTR/2023/001324).\\

\textbf{Data Availability:} No data was used for the research described in this article.\\ 

\textbf{Conflict of interest:} The authors declare that they have no conflict of interest.

\end{document}